\documentclass[11pt]{amsart}

\usepackage{amssymb}	
\usepackage{tikz}
\usepackage{tikz-cd}

\usepackage{amsmath}
\usepackage{mathrsfs}
\usepackage{amsthm}
\usepackage{verbatim}
\usepackage{subeqnarray}
\usepackage{graphicx}
\usepackage{cancel}
\usepackage{color}
\usepackage{float}
\usepackage{bm}
\usepackage{hyperref}
\usepackage{amsfonts}
\usepackage{amscd}
\usepackage{cases}
\usepackage{xcolor}
\usepackage{eucal}

\definecolor{red}{rgb}{0.8,0,0}
\definecolor{darkorange}{rgb}{1,0.4,0}
\definecolor{lightorange}{rgb}{1,0.6, 0}
\definecolor{yellow}{rgb}{1,0.8, 0}

\newtheorem{theorem}{Theorem}
\newtheorem{remark}{Remark}

\newtheorem{corollary}{Corollary}
\newtheorem{lemma}{Lemma}

\newtheorem{proposition}{Proposition}

\newcommand\tr{\operatorname{tr}}
\newcommand\inc{\operatorname{inc}}
\newcommand\skw{\operatorname{skw}}

\newcommand\vskw{\operatorname{vskw}}
\newcommand\mskw{\operatorname{mskw}}

\newcommand\sym{\operatorname{sym}}
\newcommand\grad{\operatorname{grad}}
\newcommand\deff{\operatorname{def}}
\renewcommand\div{\operatorname{div}}
\renewcommand\ker{\mathcal{N}}

\newcommand\curl{\operatorname{curl}}

\newcommand\dev{\mathrm{dev}}

\newcommand\alt{\mathrm{Alt}}

\newcommand\hess{\operatorname{hess}}

\newcommand\ran{\mathcal{R}}

\newcommand\supp{\operatorname{supp}}

\newcommand\K{\mathbb{K}}
\newcommand\M{\mathbb{M}}

\renewcommand\S{{\mathbb S}}

\newcommand\R{\mathbb{R}}

\newcommand\V{{\mathbb{V}}}

\usepackage{chngpage}

\usepackage{lipsum}

\usepackage{booktabs}

\newcommand{\bs}{{\scriptscriptstyle \bullet}}

\makeatletter
\@addtoreset{equation}{section}
\makeatother

\usepackage{geometry}
\begin{document}
\title{Bounded Poincar\'e operators for twisted and BGG complexes}

\author{Andreas \v{C}ap \and  Kaibo Hu}
\address{Faculty of Mathematics, University of Vienna, Oskar-Morgenstern-Platz 1, 1090 Wien, Austria}
\email{Andreas.Cap@univie.ac.at}
 \address{Mathematical Institute,
University of Oxford, Andrew Wiles Building, Radcliffe Observatory Quarter, 
Oxford, OX2 6GG, UK}
\email{Kaibo.Hu@maths.ox.ac.uk}
\date{\today.}
\maketitle

\begin{abstract}
Nous construisons des opérateurs de Poincaré bornés pour les complexes tordus et les complexes BGG pour une large classe d'espaces fonctionnels (par exemple, les espaces de Sobolev) sur des domaines de Lipschitz bornés. Ces opérateurs sont dérivés des versions de de~Rham en utilisant les diagrammes BGG et, pour une cohomologie triviale, satisfont à l'identité d'homotopie $dP+Pd=I$ pour des degrés $>0$. Les opérateurs préservent les classes polynomiales si les versions de~Rham le font. Le cas d'une cohomologie non triviale ainsi que la propriété de complexe $P\circ P=0$ peuvent être incorporés. Nous présentons des applications aux suites exactes de polynômes.
\end{abstract}

\begin{abstract}
We construct bounded Poincar\'e operators for twisted complexes and BGG complexes with a wide class of function classes (e.g., Sobolev spaces) on bounded Lipschitz domains. These operators are derived from the de~Rham versions using BGG diagrams and, for vanishing cohomology, satisfy the homotopy identity $dP+Pd=I$  in degrees $>0$. The operators preserve polynomial classes if the de~Rham versions do so. Nontrivial cohomology and the complex property $P\circ P=0$ can be incorporated. We present applications to polynomial exact sequences.
\end{abstract}
\smallskip
\noindent \textbf{Keywords.} Bernstein-Gelfand-Gelfand construction, Hilbert complexes, Poincar\'e operators, polynomial complexes

\noindent \textbf{MSC Codes.} 58J10, 65N30

\section{Introduction}

Poincar\'e operators are a basic tool in the theory of complexes. In the simplest case, they can be used to prove vanishing cohomology of a complex $(V^{\bs}, D^{\bs})$. For this purpose, one defines a graded operator $P^{\bs}$ such that $P^k:D^k\to D^{k-1}$ and such that $D^{k-1}P^{k}+P^{k+1}D^{k}=I$ for $k>0$  together with a simple condition in degree zero. In the case of non-trivial cohomology, a Poincar\'e operator determines a subcomplex with isomorphic cohomology. For this case, the above relation has to be modified to $D^{k-1}P^{k}+P^{k+1}D^{k}=I-L^k$ for operators $L^k:D^k\to D^k$ with favorable properties. Poincar\'e operators for the de~Rham complex play a key role in various applications.  In manifold theory, the existence of Poincar\'e operators on star-shaped domains immediately implies local exactness of the de~Rham complex. In PDE theories, the Poincar\'e operators serve as the homotopy inverse of differential operators and thus imply the existence of solutions for, e.g., the Stokes problem. See, e.g., \cite{acosta2017divergence,danchin2012divergence}, for the inverse of the divergence operator. For smooth de~Rham complexes, the Poincar\'e operators also form a complex $P^{k-1}\circ P^{k}=0$ and are polynomial-preserving. These properties are important for constructing finite element spaces \cite{arnold2006defferential,christiansen2018generalized,christiansen2019finite,hiptmair1999canonical,hu2022oberwolfach,hu2020simple,hu2022family} and establishing $p$-robustness for high order methods (see, e.g.,  \cite{bespalov2010convergence,ern2015polynomial}).

Many of these applications require Poincar\'e operators that are bounded between certain function spaces (e.g., Sobolev spaces).  Costabel and McIntosh \cite{costabel2010bogovskiui} established such operators and showed that these operators are pseudo-differential operators of order -1 and thus are bounded in a broad class of function spaces. The regularized Poincar\'e and Bogovski\u{\i} operators are special cases of this general construction. Consequences of the construction in \cite{costabel2010bogovskiui} include a number of analytic results, such as the Poincar\'e inequalities, the existence of regular potential and compactness \cite[Section 2]{arnold2021complexes}.  From explicit Poincar\'e operators,  one may also estimate constants in inequalities. We refer to \cite{guzman2020estimation} for work in this direction.

In addition to the de~Rham complexes, other differential complexes encode structures in various problems in continuum mechanics   \cite{amstutz2019incompatibility,angoshtari2015differential,kroner1995dislocations} and general relativity \cite{beig2020linearised,quenneville2015new}. The Bernstein-Gelfand-Gelfand (BGG) machinery provides a systematic way to construct such complexes \cite{arnold2021complexes,vcap2022bgg,vcap2001bernstein,eastwood2000complex}. This includes the elasticity complex as an example. Analytic results of Costabel-McIntosh's type were established in \cite{arnold2021complexes,vcap2022bgg}  for these complexes. The results were directly carried over from the de~Rham complexes via BGG diagrams, which avoided the use of Poincar\'e operators. Nevertheless, inspired by applications associated with the de~Rham complexes, explicit Poincar\'e operators are still in need for these complexes, as they
\begin{itemize}
\item establish exact sequences with polynomial coefficients, which are crucial for finite element computation;
\item facilitate the construction and analysis of $p$-robust numerical methods (algorithms that are uniformly robust with polynomial degree);
\item contain explicit constants, such that refined estimates of the constants are possible;
\item provide tools for establishing mechanics models (e.g., applications of the  Ces\`aro-Volterra formula in intrinsic elasticity \cite{ciarlet2009intrinsic}).
\end{itemize}
For the elasticity complex with smooth functions on domains with trivial topology, Poincar\'e operators were derived in \cite{christiansen2020poincare}. 
The operator at index one coincides with the Ces\`aro-Volterra formula in classical elasticity. However, these operators are not bounded between function spaces involved in PDEs and numerical analysis. This makes it impossible to directly generalize several of the applications discussed for the de Rham complex above to tensor-valued problems such as elasticity and relativity. 


In this paper, we fill this gap by deriving bounded Poincar\'e operators for BGG complexes. The basic idea is to follow the BGG diagrams and use complex maps and a diagram chase to obtain the operators for the BGG complexes from the de~Rham versions. The strategy is a generalization of the special case in  \cite{christiansen2020poincare}. We show that the construction can be applied to the BGG complexes constructed in \cite{vcap2022bgg} from representation theory input, where the two steps in the construction need different regularity assumptions. In particular, the $K$ operators in the BGG diagram that were a crucial tool in \cite{christiansen2020poincare} are not available (as bounded operators) in the second step. Compared to  \cite{christiansen2020poincare},
\begin{itemize}
\item The Poincar\'e operators are bounded and valid for all the BGG complexes in \cite{arnold2021complexes} and \cite{vcap2022bgg}. Particularly, this works for not only BGG complexes derived from diagrams with two rows (e.g., the elasticity complex), but also for those derived from diagrams with several rows (e.g., the conformal deformation complex).
\item For the examples from \cite{vcap2022bgg} coming from representation theory, we show that the derived operators are polynomial-preserving. 
\item We develop a general algebraic strategy to modify the derived operators such that the complex property $\mathscr{P}^{k-1}\circ \mathscr{P}^{k}=0$ holds.
\item The work \cite{christiansen2020poincare} assumes trivial cohomology, while in this paper, we discuss the incorporation of nontrivial cohomology.
\end{itemize}

An intermediate step in the construction of BGG complexes is the twisted complexes, which are important in their own right. In geometry, the twisted complex can be viewed as providing a prolongation of the operators of geometric origin that show up in the BGG complexes, (c.f.\ the interpretation of the elasticity complex as the Riemannian deformation complex). In continuum mechanics, the twisted complexes can be interpreted as encoding additional degrees of freedom compared to the BGG complexes.  More specifically, from one to three space dimensions, Hodge-Laplacian problems of an appropriate twisted de~Rham complex correspond to the linearized Timoshenko beam, the Reissner-Mindlin plate, and the Cosserat elasticity models, respectively; the resulting BGG complexes correspond to the linearized Euler-Bernoulli beam, the Kirchhoff-Love plate, and the standard elasticity, respectively. See \cite[Section 5]{vcap2022bgg} for more details. In this paper, we also derive bounded Poincar\'e operators for the twisted complexes. Existing results for the elasticity version of the twisted complex in the context of Cosserat continuum can be found in \cite{gunther1958statik}.

The Ces\'aro-Volterra formula has been generalized to the case with ``little regularity'' in  \cite{ciarlet2010cesaro}. The construction in \cite{ciarlet2010cesaro} leads to an implicit form based on duality, while the corresponding result in this paper has a more explicit form, preserves polynomial classes and is valid for a wider range of function spaces. As a further application of the new Poincar\'e operators, we construct exact BGG complexes with polynomial coefficients. 



We summarize the main results in this paper using a specific example, i.e., the elasticity complex in 3D, which has the form (see Section \ref{sec:de Rham} for the notation):
\begin{equation}\label{elasticity-q}
\begin{tikzcd}
0\arrow{r}& H^{q}(\Omega)\otimes \mathbb{V}\arrow{r}{\deff} &  H^{q-1}(\Omega)\otimes \mathbb{S} \arrow{r}{\inc} &  H^{q-3}(\Omega)\otimes \mathbb{S} \arrow{r}{\div} &H^{q-4}(\Omega)\otimes \mathbb{V}\arrow{r} & 0,
\end{tikzcd}
\end{equation}
Similar conclusions hold for other complexes, e.g., the strain and stress complexes in 2D, and the Hessian, $\div\div$, conformal Hessian, and conformal deformation complexes in 3D \cite{arnold2021complexes,vcap2022bgg}.
\begin{theorem}\label{thm:main-example}
Let $\Omega$ be a bounded Lipschitz domain.
For any real number $q$, there exist bounded operators $\mathscr{P}^{i}, ~i=1, 2, 3$ that fit in the following sequence:
\begin{equation} \label{seq:P}
\begin{tikzcd}[swap]
0&\arrow{l} H^{q}(\Omega)\otimes \mathbb{V}&\arrow{l}{\mathscr{P}^{1}}  H^{q-1}(\Omega)\otimes \mathbb{S}&\arrow{l}{\mathscr{P}^{2}}  H^{q-3}(\Omega)\otimes \mathbb{S} &\arrow{l}{\mathscr{P}^{3}}H^{q-4}(\Omega)\otimes \mathbb{V}&\arrow{l}  0,
\end{tikzcd}
\end{equation}
  such that the following homotopy relations hold: 
\begin{align*}
\mathscr{P}^{1}\deff u &=(I-\mathscr{L}_{V}^{0})u, \quad \forall u\in H^{q}\otimes \mathbb{V},\\
(\deff\mathscr{P}^{1}+\mathscr{P}^{2}\inc) e&=(I-\mathscr{L}_{V}^{1})e,   \quad \forall e\in H^{q-1}\otimes \mathbb{S},\\
(\inc\mathscr{P}^{2}+\mathscr{P}^{3}\div) \sigma&=(I-\mathscr{L}_{V}^{2})\sigma,   \quad \forall \sigma\in H^{q-3}\otimes \mathbb{S},\\
\div  \mathscr{P}^{3}v&=(I-\mathscr{L}_{V}^{3})v,   \quad \forall  v\in H^{q-4}\otimes \mathbb{V},
\end{align*}
where $\mathscr{L}_{V}^{i}, ~i=0, 1, 2, 3$ are induced by pseudodifferential operators of order $-\infty$ and hence are smoothing operators.

If   $\Omega$ is  a contractible domain,  then ${\mathscr{L}}_{V}^{k}=0$ for $k>0$, while the cohomology in degree $0$ is the space $\mathrm{RM}:=\{a+b\wedge x: a, b\in \mathbb{R}^{3}\}$ of infinitesimal rigid body motions.  If $\Omega$ is further starlike, one can choose the operators $\mathscr{P}^{\bs}$ such that the complex property $\mathscr{P}\circ \mathscr{P}=0$ holds, i.e., \eqref{seq:P} is a complex. The operators $\mathscr{P}^{\bs}$ preserve polynomial classes, i.e., the operators fit in the following sequence for any $r\geq 0$:
\begin{equation} \label{seq:Hr}
\begin{tikzcd}[swap]
0&\arrow{l} \mathcal{H}_{r+4} \otimes \mathbb{V}&\arrow{l}{\mathscr{P}^{1}} \mathcal{H}_{r+3}\otimes \mathbb{S}&\arrow{l}{\mathscr{P}^{2}}  \mathcal{H}_{r+1}\otimes \mathbb{S} &\arrow{l}{\mathscr{P}^{3}}\mathcal{H}_{r}\otimes \mathbb{V}&\arrow{l}  0,
\end{tikzcd}
\end{equation}
where $\mathcal{H}_{s}$ is the space of homogeneous polynomials of degree $s$. 
\end{theorem}

The rest of the paper is organized as follows. In Section \ref{sec:de Rham}, we review basic facts on Poincar\'e operators as well as the results on existence of bounded Poincar\'e operators for de~Rham complexes. In Section \ref{sec:construction}, we present the main algebraic construction for the BGG versions in a general setting and assuming trivial cohomology. In Section \ref{sec:examples}, we first show that our construction can be applied to the examples from \cite{vcap2022bgg} derived from representation theory. We prove that in these cases the resulting operators are polynomial preserving and that the complex property can be achieved. Then we discuss several explicit examples of the construction.  In Section \ref{sec:modifications}, we discuss the changes needed to deal with non-vansihing cohomology. Finally, vector/matrix proxies and an application to polynomial exact sequences are presented in Section \ref{sec:applications}.

\section{Review: Poincar\'e operators and de~Rham complexes}\label{sec:de Rham}

In this section, we review Poincar\'e operators, in particular various versions for de Rham complexes. The results will be a crucial ingredient in the construction of the BGG version. Moreover, the applications will serve as a motivation for the development and fix the notation used in the rest of this paper. 

Let $V^{i}, i=0, 1, \cdots, n$ be  vector spaces and $d^{i}: V^{i}\to V^{i+1}, i=0, 1, \cdots, n-1$ be linear operators. Then a sequence 
\begin{equation*}
\begin{tikzcd}
\cdots\arrow{r}{}&V^{i-1}\arrow{r}{d^{i-1}} &V^{i} \arrow{r}{d^{i}}&V^{i+1}\arrow{r}{ }&\cdots
 \end{tikzcd}
\end{equation*}
is called a complex if $d^{i+1}\circ d^{i}=0,~ \forall i$. A consequence of the complex property is $\ran(d^{i-1})\subset \ker(d^{i})$. Cohomology is the quotient space $\mathscr{H}^{i}:=\ker(d^{i})/\ran(d^{i-1})$. Suppose that there exists linear maps $P^{k}: V^{k}\to V^{k-1}, k=1, 2, \cdots, n$, such that 
\begin{equation}\label{DPPD-general}
d^{k-1}P^{k}+P^{k+1}d^{k}=I_{V^{k}}  \text{\ for\ } k>0 \text{\ and\ } d^0P^1d^0=d^0,  
\end{equation}
then we call $P^{\bs}$ an (algebraic) Poincar\'e operator. We will call \eqref{DPPD-general} a null-homotopy relation. Note that \eqref{DPPD-general} for $k>0$ immediately implies the exactness of the complex and hence vanishing cohomology. Indeed, for any $u\in V^{k}$ with $du=0$, we get $u=d(Pu)$, so the class of $u$ in $\mathscr{H}^k$ vanishes. The condition in degree $0$ readily implies that for $u\in V^0$, $u-Pdu$ lies in the kernel of $d$, which coincides with the cohomology in degree zero. To allow for nontrivial cohomology, one can weaken the relation \eqref{DPPD-general} to the form
\begin{equation}\label{DPPD-L}
d^{k-1}P^{k}+P^{k+1}d^{k}=I_{V^{k}}-L^{k}, 
\end{equation}
for a family $L^\bs$ of linear maps $L^k:V^k\to V^k$. Observe that the the relation \eqref{DPPD-L} readily implies that $d^kL^k=d^kP^{k+1}d^k-d^k=L^{k+1}d^k$. Hence defining $W^k:=\ran(L^k)\subset V^k$, we conclude that $d^k(W^k)\subset W^{k+1}$, so this defines a subcomplex $(W^\bs, d^\bs)$ of $(V^\bs,d^\bs)$. The inclusions $W^k\to V^k$ induce a map between the cohomologies of the two complexes and the maps $L^k:V^k\to W^k$ define a map in the opposite direction. The relation \eqref{DPPD-L} easily implies that these maps induce isomorphisms in cohomology which are inverse to each other. 

There are different choices for $P^{\bs}$ and $L^{\bs}$ such that \eqref{DPPD-L} holds (actually, once we choose $P^{\bs}$, \eqref{DPPD-L} defines $L^{\bs}$). We should specify what properties we want $L^{\bs}$ to satisfy. Following \cite{costabel2010bogovskiui}, in this paper we are mainly concerned with the case where $L^{\bs}$ are smoothing operators. Before presenting details of this choice, we mention another possibility. In fact, the most efficient version of such a relation is when the resulting differentials on the complex $W^\bs$ are the zero maps, i.e.~if $d^kL^k=0$ for any $k$. Then the cohomology of $(V^\bs,d^\bs)$ is simply given by $\mathscr{H}^{i}=\ran(L^i)$.  Observe also, that starting from \eqref{DPPD-L} the additional conditions in \eqref{DPPD-general} are equivalent to $dL^0=0$ and $L^k=0$ for $k>0$. 

We have to remark that in the setting of the current article we will rarely meet the situation of bounded Poincar\'e operators which lead to $L$-operators that satisfy the strict relation $d^kL^k=0$. However, this case is very important in different situations, whence we decided to also discuss results on this case. The prime example for this situation comes from Hodge theory on smooth differential forms on compact Riemannian manifolds. Under this assumption, one in addition to the exterior derivative $d:\Omega^k(M)\to\Omega^{k+1}(M)$ obtains the codifferential $\delta:\Omega^k(M)\to\Omega^{k-1}(M)$ and hence the Hodge-Laplace operator $\Delta:=d\delta+\delta d$. Now Hodge theory states that for compact $M$ and each degree $k$ the kernel of $\Delta$ on $\Omega^k(M)$ equals $\ker(d)\cap\ker(\delta)$ and isomorphically projects onto the cohomology in degree $k$. The key towards proving this is the construction of a bounded operator (called Green's operator) $G:\Omega^k(M)\to\Omega^k(M)$ for each $k$ such that $\ker(G)=\ker(\Delta)$, $Gd=dG$ and $\delta G=G\delta$ and such that $G\Delta=I-L$, where $L$ is a projection onto $\ker(\Delta)$. In fact, $G$ can be defined by solving a Hodge-Laplacian problem
$$
(d\delta +\delta d)G(u):=P_{\mathfrak{H}^{\perp}} u,
$$
where $P_{\mathfrak{H}^{\perp}}$ is the projection to the ($L^{2}$) orthogonal complement of harmonic forms.
 This shows that defining $P:=\delta G$, one obtains $dP+Pd=G\Delta=I-L$ and of course $dL=0$.

 The main facts from Hodge theory can be generalized to the setting of complexes of densely defined unbounded operators between Hilbert spaces, see \cite{arnold2018finite,Arnold.D;Falk.R;Winther.R.2006a,Arnold.D;Falk.R;Winther.R.2010a}. The BGG machinery can also be applied in that setup and then the above discussion should apply to the twisted complex and BGG complexes obtained in that setting.

\medskip

For the case of the de Rham complex, we consider a domain $\Omega\subset\mathbb R^n$ (conditions on $\Omega$ will be imposed later). Then there is the smooth de Rham complex
\begin{equation*}
\begin{tikzcd}
0\arrow{r}& C^{\infty}\Lambda^{0} \arrow{r}{d^{0}}&C^{\infty}\Lambda^{1} \arrow{r}{d^{1}}&\cdots \arrow{r}{d^{n-1}}&C^{\infty}\Lambda^{n}\arrow{r}& 0.
\end{tikzcd}
\end{equation*}
  Here $V^{k}:=C^{\infty}\Lambda^{k}$ is the space of smooth differential $k$-forms and $d^{k}$ is the $k$-th exterior derivatives. Assuming that the domain $\Omega$ is star shaped around a point $x_0\in\Omega$, one can explicitly define a Poincar\'e operator. For $t\in [0,1]$, one defines a map $F_t:\Omega\to\Omega$ by $F_t(x)=tx+(1-t)x_0$. Hence $F_{1}$ is identity and $F_{0}$ is the  constant map $x_{0}$. Now for a $k$-form $u$ on $\Omega$, define 
\begin{equation}\label{smooth-poincare}
(\mathfrak{p} u)_x(\xi_2, \ldots, \xi_k) = \int_0^1 u_{F_t(x)}(\partial_tF_t(x), D F_t(x)\xi_2, \cdots, D F_t(x) \xi_k) \,d t.
\end{equation}
Then one verifies that $d\mathfrak{p}+\mathfrak{p}d=I$ for $k\geq 1$ and $ \mathfrak{p}du(x)=u(x)-u(x_{0})$ for $k=0$ \cite{cartan12701formes,christiansen2018generalized,lang2012fundamentals}. The existence of the $u(x_{0})$ term reflects the fact that the de~Rham complex is not exact at index zero ($d^{0}$ has constants as its kernel). This Poincar\'e operator is the standard tool for proving the Poincar\'e lemma (local exactness of the de~Rham complex) \cite{lang2012fundamentals}. Of course, if $\Omega$ has nontrivial cohomology then a Poincar\'e operator that satisfies \eqref{DPPD-general} cannot exist. 

For many applications, versions of the de Rham complex with non-smooth differential forms are needed. To fix the notation, let $\Omega$ be a bounded Lipschitz domain. Then we let $H^{q}(\Omega)$ be the Sobolev space with all derivatives of order  $\leq q$ in $L^{2}(\Omega)$, and $H_{0}^{q}(\Omega)$ be the closure of $C^{\infty}_{0}(\Omega)$ with the $H^{q}$ norm.    Moreover,  $H^{q}\Lambda^{k}(\Omega)$ is the space of $k$-forms with coefficients in $H^{q}$. Similarly, one can define $H^{q}_{0}\Lambda^{k}(\Omega)$. We also omit the domain $\Omega$ when there is no danger of confusion.  In the sequel, we will use $\|\cdot\|_{q}$ to denote the Sobolev $H^{q}$ norm, and $\|\cdot\|$ for the $L^{2}$ norm.

The following de Rham complex with Sobolev spaces plays an important role in the analysis of fluid and electromagnetic problems:
\begin{equation}\label{deRham-Hq}
\begin{tikzcd}
0\arrow{r}& H^{q}\Lambda^{0} \arrow{r}{d^{0}}&H^{q-1}\Lambda^{1} \arrow{r}{d^{1}}&\cdots \arrow{r}{d^{n-1}}&H^{q-n}\Lambda^{n}\arrow{r}& 0. 
\end{tikzcd}
\end{equation}
For the analytical applications, it is desirable to find a Poincar\'e operator such that each of the operators $P^i$ is bounded. The operators \eqref{smooth-poincare} cannot be used for \eqref{deRham-Hq} as they are not bounded between the Sobolev spaces.
Costabel and McIntosh \cite{costabel2010bogovskiui} derived a bounded Poincar\'e operator for \eqref{deRham-Hq} on bounded Lipschitz domains. This starts from the case of a domain $\Omega$ that is starlike with respect to all points in a ball $B$. Then one chooses a bump function (a non-negative function whose integral is $1$) $\theta$ with support in $B$ and averages the operators coming from \eqref{smooth-poincare} for the points $x_0\in B$ with a weight $\theta(x_0)$. It turns out that the resulting operators $P^\bs$ still satisfy the homotopy relation \eqref{DPPD-general}, but in addition, they are pseudodifferential operators of order $-1$ on $\mathbb R^n$. This implies favorable analytical properties in a wide range of spaces, e.g.,  $W^{k, p}$, Besov spaces and Triebel-Lizorkin spaces. For simplicity, we focus on Sobolev spaces $H^q$, for which obtains bounded operators $P^{i}: H^{q-i}  \Lambda^{i}\to H^{q-(i-1)}  \Lambda^{i-1}$. Moreover, the operators turn out to not only map smooth forms to smooth forms but also preserve spaces of forms with appropriate polynomial coefficients.  

A bounded Lipschitz domain $\Omega\subset\mathbb R^n$ then admits a finite covering by domains that are starlike with respect to all points in some open ball. One then glues together the Poincar\'e operators for these starlike domains using a partition of unity. It turns out that the resulting operators $P^\bs$ still are pseudodifferential operators of order $-1$, so analytical properties like boundedness properties are still available. They satisfy homotopy relations of the form \eqref{DPPD-L}, but the properties of the resulting operators $L^\bs$, which are caused by the glueing with a partition of unity, are slightly unexpected. They are not finite rank operators (so their images are much larger than the cohomology of the de Rham complex) but they are pseudodifferential operators of order $-\infty$ and thus smoothing operators. So additional work is needed to complete the analysis of the Sobolev de Rham complex \eqref{deRham-Hq}. This is based on the fact that, since $L^\bs$ is a smoothing operator, $I-L$ is invertible modulo compact operators. We remark that in \cite{costabel2010bogovskiui} also the compactly supported version, i.e.\ the Hodge duals of Poincar\'e operators are studied under the name {\it generalized Bogovski{\u\i} operators}. 


The Poincar\'e operators for \eqref{deRham-Hq} have several consequences and applications. 
\begin{itemize}
\item The complex \eqref{deRham-Hq} has a uniform representation of cohomology for any real number $q$, i.e., 
$$
\ker(d^{k}, H^{q-k}\Lambda^{k})=d^{k-1}H^{q-(k-1)}\Lambda^{k-1}\oplus \mathscr{H}^{k},
$$
where $\mathscr{H}^{k}$ is a linear subspace of $C^{\infty}\Lambda^{k}$ that is independent of $q$, see \cite{costabel2010bogovskiui}. Moreover, $\dim \mathscr{H}^{k}=\beta_{k}$, the $k$-th Betti number. The uniform representation of cohomology further implies analytical results, e.g., the Poincar\'e inequalities, regular decomposition and compactness \cite{arnold2021complexes}. 
\item The special case of $P^{n}$ implies that for any $q\in L^{2}$ satisfying $\int_{\Omega}q\,dx=0$, there exists $v\in [H_{0}^{1}]^{n}$ such that $\div v=q$, and $\|v\|_{1}\leq \|q\|$. This establishes the well-posedness of the Stokes problem (see, e.g., \cite{Girault.V;Raviart.P.1986a,temam1984navier}): given $f\in H^{-1}$, find $u\in [H_{0}^{1}]^{n}$ and $p\in L^{2}/\mathbb{R}$, such that
\begin{eqnarray*} 
-\Delta u+\nabla p &=& f,\\
\div u&=&0.
\end{eqnarray*}
\item The Poincar\'e operators can be used to construct polynomial complexes and show their exactness, which are needed for finite element methods. Examples include
\begin{equation}\label{polynomial-1}
\begin{tikzcd}
\cdots \arrow{r}& \mathcal{P}_{r-(k-1)}\Lambda^{k-1} \arrow{r}{d^{k-1}}& \mathcal{P}_{r-k}\Lambda^{k}  \arrow{r}{d^{k}}& \mathcal{P}_{r-(k+1)}\Lambda^{k+1}  \arrow{r}{}&\cdots, 
\end{tikzcd}
\end{equation}
and
\begin{equation}\label{polynomial-2}
\begin{tikzcd}
\cdots \arrow{r}& \mathcal{P}_{r }\Lambda^{k-1}+ \mathfrak{p}^{k}\mathcal{P}_{r }\Lambda^{k} \arrow{r}{d^{k-1}}& \mathcal{P}_{r }\Lambda^{k}+ \mathfrak{p}^{k+1}\mathcal{P}_{r }\Lambda^{k+1}\arrow{r}{d^{k}}& \mathcal{P}_{r }\Lambda^{k}+ \mathfrak{p}^{k+1}\mathcal{P}_{r }\Lambda^{k+1}  \arrow{r}{}&\cdots, 
\end{tikzcd}
\end{equation}
where $ \mathcal{P}_{r }\Lambda^{k}$ denotes the space of $k$-forms with polynomial coefficients of degree less than or equal to $r$ \cite{arnold2018finite,Arnold.D;Falk.R;Winther.R.2006a,hiptmair1999canonical}. See also, e.g., \cite{christiansen2018generalized,hu2020simple,hu2022family}, for applications to the construction of finite elements for Stokes problems.  The Poincar\'e operators \eqref{smooth-poincare} for the smooth de~Rham complex are sufficient for this purpose.
\item The bounded Poincar\'e operators \cite{costabel2010bogovskiui} for \eqref{deRham-Hq} preserve certain function classes, particularly polynomials. This fact is used to show the $p$-robustness (uniformity with respect to polynomial degrees) of high order finite element methods. For this application, it is not enough to know the exactness (cohomology) of the Sobolev de~Rham complexes, but the explicit form of the Poincar\'e operators is important.
\end{itemize}

For $n\leq 3$, the de~Rham complex can be written in terms of vector-matrix proxies on $\mathbb{R}^{n}$. This means that we can choose coordinates and identify differential forms with the (function) coefficients in front of the form bases (see \cite{arnold2018finite,Arnold.D;Falk.R;Winther.R.2006a} for more details). In particular, the proxy of the de~Rham complex in three dimensions involves $\grad$, $\curl$ and $\div$ operators:
\begin{equation*} 
\begin{tikzcd}
0\arrow{r}& C^{\infty}  \arrow{r}{\grad}&C^{\infty} \otimes \mathbb{R}^{3} \arrow{r}{\curl}&C^{\infty} \otimes\mathbb{R}^{3}\arrow{r}{\div}&C^{\infty} \arrow{r}& 0. 
\end{tikzcd}
\end{equation*}
Correspondingly, the Poincar\'e operators have the following vector-matrix proxies (without loss of generality, we have chosen 0 as the base point and the path $\gamma(t)=tx$ connecting $0$ and $x$; see, e.g., \cite[(7)-(8)]{hiptmair1999canonical}, \cite[(34)]{christiansen2018generalized}):
\begin{align}\label{P-3D}
P^{1}u(x)=\int_{0}^{1}u(tx)\cdot x\,dt, \quad P^{2}v(x)=\int_{0}^{1}tv(tx)\wedge x dt, \quad P^{3}w(x)=\int_{0}^{1}t^2w(tx)x \,dx.
\end{align}
In the proxies below, we use the following convention of notation:
\begin{equation}\label{dy-dot}
\int_{0}^{x}u(y)\cdot dy:=\int_{\gamma}u(y)\,dy=\int_{0}^{1}u(tx)\cdot x\,dt,
\end{equation}
where $\gamma(t):=tx$ for $t\in [0, 1]$ is as above. Similarly, 
 for smooth functions   $v=(v^{1}, v^{2}, v^{3})$ and $w$, 
\begin{equation}\label{dy-wedge}
\int_{0}^{x}dy\wedge v(y)=-\int_{0}^{x} v(y)\wedge dy:=\left (
\begin{array}{c}
\int_{0}^{1}tv_{3}(tx)x^{2}\,dt-\int_{0}^{1}tv_{2}(tx)x^{3}\,dt\\
\int_{0}^{1}tv_{1}(tx)x^{3}\,dt-\int_{0}^{1}tv_{3}(tx)x^{1}\,dt\\
\int_{0}^{1}tv_{2}(tx)x^{1}\,dt-\int_{0}^{1}tv_{1}(tx)x^{2}\,dt
\end{array}
\right ), 
\end{equation}
\begin{equation}\label{dy-tensor}
\int_{0}^{x} w(y)\otimes dy=\int_{0}^{x}dy\otimes w(y):=\int_{0}^{1}t^2w(tx)x\,dt.
\end{equation}

In this notation, \eqref{P-3D} becomes
\begin{align}\label{P-3D-new-notation}
P^{1}u(x)=\int_{0}^{x}u(y)\cdot dy, \quad P^{2}v(x)=\int_{0}^{x}v(y)\wedge dy, \quad P^{3}w(x)=\int_{0}^{x}w(y)\otimes dy.
\end{align}

 These operators are not bounded between the desired Sobolev spaces because the integral operators involve fixed base points.  

In elasticity, the Ces\`aro-Volterra integral operator provides an explicit potential for strain tensors satisfying the Saint-Venant compatibility condition. In terms of the elasticity complex, this can be viewed as a Poincar\'e operator in the elasticity complex at index one
\begin{equation}\label{elasticity-smooth}
\begin{tikzcd}
0\arrow{r}& C^{\infty}\otimes \mathbb{V}\arrow{r}{\deff} &  C^{\infty}\otimes \mathbb{S} \arrow{r}{\inc} & C^{\infty}\otimes \mathbb{S} \arrow{r}{\div} &C^{\infty}\otimes \mathbb{V}\arrow{r} & 0,  
\end{tikzcd}
\end{equation}
where $\deff u:=\frac{1}{2}(\nabla u+\nabla^{T}u)$ is the symmetric gradient, $\inc u:=\curl\circ \mathrm{T}\circ \curl u$ is the incompatibility operator (the Saint-Venant operator, or the linearized Einstein tensor up to a constant coefficient), and $\div$ is the column-wise divergence.  Here and in the sequel, $\mathbb V:=\mathbb R^n$ denotes vectors and $\mathbb M$ is the space of all $n\times n$-matrices. Let further $\mathbb S$ and $\mathbb T$ be the subspaces of matrices that are symmetric and trace-free, respectively.

 

\section{Poincar\'e operators for BGG complexes with trivial cohomology}\label{sec:construction}

In this section, we derive bounded Poincar\'e operators for the BGG complexes in the general (abstract) setting used in  \cite{vcap2022bgg}. For different parts of the construction in that reference different assumptions are needed, so we have to connect to the right point in the development. 

Let $Z^{i, j}, ~0\leq i\leq n, 0\leq j\leq N$ be Hilbert spaces and $d^{i, j}: Z^{i, j}\to Z^{i+1, j}$ be bounded linear operators such that for each $j$ we obtain a complex $(Z^{\bs,j},d^{\bs,j})$, i.e.~$d^{i+1,j}\circ d^{i,j}=0$ for all $i,j$. Assume further that we have given bounded linear operators $S^{i, j}: Z^{i, j}\to Z^{i+1, j-1}$ satisfying $d^{i+1, j-1}S^{i, j}=-S^{i+1, j}d^{i, j}$, $S^{i+1, j-1}\circ S^{i, j}=0, \forall i, j$ (or, in short,  $dS=-Sd$ and $S^{2}=0$). At this stage, we don't need or use the operators called $K^{i,j}$ in \cite{vcap2022bgg}.

Let $Y^{i}:=Z^{i, 0}\times \cdots \times Z^{i, N}$. Following the convention in \cite{vcap2022bgg}, we use $S^{i}: Y^{i}\to Y^{i+1}$ to denote the vectorized operator, i.e., for any $u^{i, j}\in Z^{i, j}$,
 $$
 S^{i}(u^{i, 0}, u^{i, 1}, \cdots, u^{i, N}):=
 \left ( 
 \begin{array}{ccccc}
 0 & S^{i, 1} & &&\\
  &0 & S^{i, 2} & &\\
  & & \ddots &\ddots &\\
    &  &  &0 &S^{i, n} \\
  &  & & &0
 \end{array}
 \right )
 \left (
 \begin{array}{c}
 u^{i, 0}\\ u^{i, 1}\\ \vdots\\ u^{i, N-1} \\u^{i, N}
 \end{array}
 \right )
 = (S^{i, 1}u^{i, 1}, \cdots, S^{i, N}u^{i, N}, 0).
 $$
 These ingredients are sufficient to define the \textit{twisted complex}: Defining the twisted operator $d_{V}:=d-S$, the twisted complex is
 \begin{equation}\label{A-sequence-0}
\begin{tikzcd}
\cdots \arrow{r}{} &Y^{i-1} \arrow{r}{d_{V}^{i-1}} &Y^{i} \arrow{r}{d_{V}^{i}} & Y^{i+1}  \arrow{r}{} & \cdots,
 \end{tikzcd}
 \end{equation}
Also, the construction of the BGG complex from the twisted complex \cite{vcap2022bgg} can be carried out with the ingredients we have at hand so far, the $K$ operators are not needed in this construction. 

Now we assume in addition that for each of the initial complexes $(Z^{\bs,j},d^{\bs,j})$, we have given bounded Poincar\'e operators. This means that for each $i,j$, we have given a bounded operator $P^{i, j}: Z^{i, j}\to Z^{i-1, j}, i=1, 2, \cdots, n$, such that 
\begin{equation}\label{DPPD-input}
d^{i-1, j}P^{i, j}+P^{i+1, j}d^{i, j}=I_{Z^{i, j}}  \text{\ for\ }i>0\text{\ and\ }d^{0,j}P^{1,j}d^{0,j}=d^{0,j}.
\end{equation}
As noted in Section \ref{sec:de Rham} already, this implies that the cohomologies of all the complexes are trivial. In typical examples, each of the complexes $(Z^{\bs,j},d^{\bs,j})$ is a vector-valued de Rham complex on a star shaped domain, and appropriate Poincar\'e operators are provided by \cite{costabel2010bogovskiui}. As above, we define $P^{i}: Y^{i}\to Y^{i-1}$ as the vectorized component-wise version of $P^{i, j}$.

Our setup is summarized in the following diagram:
 \begin{equation}\label{diag:multi-rows}
\begin{tikzcd}[row sep=large, column sep = huge]
0 \arrow{r} &Z^{0, 0}  \arrow[r, "d^{0, 0}", shift left=1]  &\arrow{l}{P^{1, 0}}Z^{1, 0} \arrow[r, "d^{1, 0}", shift left=1] &\arrow{l}{P^{2, 0}}\cdots \arrow[r, "d^{n-1, 0}", shift left=1] & \arrow{l}{P^{n, 0}} Z^{n, 0} \arrow{r}{} & 0\\
0 \arrow{r}&Z^{0, 1} \arrow[r, "d^{0, 1}", shift left=1]  \arrow[ur, "S^{0, 1}"]&\arrow{l}{P^{1, 1}}Z^{1, 1}  \arrow[r, "d^{1, 1}", shift left=1]  \arrow[ur, "S^{1, 1}"]&\arrow{l}{P^{2, 1}}\cdots  \arrow[r, "d^{n-1, 1}", shift left=1] \arrow[ur, "S^{n-1, 1}"]&\arrow{l}{P^{n, 1}} Z^{n, 1}\arrow{r}{} & 0\\
 & \cdots \arrow[ur, " S^{0, 2}"]&  \cdots\arrow[ur, "S^{1, 2} "]&\cdots \arrow[ur, "S^{n-1, 2} "] &\cdots & \\
 0 \arrow{r}&Z^{0, N} \arrow[r, "d^{0, N}", shift left=1]  \arrow[ur, "S^{0, N}"]&\arrow{l}{P^{1, N}}Z^{1, N}  \arrow[r, "d^{1, N}", shift left=1]  \arrow[ur, "S^{1, N}"] &\arrow{l}{P^{2, N}}\cdots  \arrow[r, "d^{n-1, N}", shift left=1] \arrow[ur, "S^{n-1, N}"] &\arrow{l}{P^{n, N}} Z^{n, N}\arrow{r}{} & 0.
 \end{tikzcd}
\end{equation}

\begin{remark}\label{rem:K}
In the following developments, the composition $PS$ will play an important role. The components of this are $P^{i+1,j-1}\circ S^{i,j}:Z^{i,j}\to Z^{i,j-1}$, so the map between the same spaces as the operators $K^{i,j}$ used in \cite{vcap2022bgg}. As we shall see in the proof of Lemma \ref{lem:I-PS}, the analogy goes further, since $d(PS)-(PS)d=S$. Hence it is tempting to view $PS$ as a ``replacement'' for the $K$ operators. This is even more tempting in the situation of diagrams with two rows from \cite{arnold2021complexes} since in that case $e^{-K}=1-K$ which looks even close to what we are doing here. One has to be aware of the fact that there are substantial and important differences between the two situations, in particular in the examples derived from de Rham complexes. In this situation the $K$ operators from \cite{vcap2022bgg} are local operators defined on all of $\Bbb R^n$ and can be used on \textit{any} domain in $\Bbb R^n$ simply by restriction. They induce isomorphisms between the sum of the initial complexes and the twisted complex on any domain, regardless of whether the topology is trivial or not. In contrast, the version of Poincar\'e operators we need here is only available for star-shaped domains and depends on the global behavior on the domain. Also the isomorphisms between the sum of the original complexes and the twisted complex is constructed in an essentially different way here. 
  \end{remark}

We will derive the Poincar\'e operators for the BGG complex in two steps. The first step is to construct bounded Poincar\'e operators for the twisted complex $(Y^{\bs}, d_{V}^{\bs})$. The second step is then to pass to the BGG complex $(\Upsilon^{\bs}, \mathscr{D}^{\bs})$. The derivation is sketched in the following diagram:
\begin{equation}\label{diagram-two-steps}
\begin{tikzcd}[swap]
Y^{i}\arrow{d}{ F} \arrow[r, "d^{i}", shift right=1] &\arrow{l}{P^{i+1}}Y^{i+1}\arrow{d}{ F}\\
Y^{i} \arrow[d, "B^{i}", shift right=1] \arrow[r, "d_{V}^{i}", shift right=1]&\arrow{l}{P_{V}^{i+1}}Y^{i+1}\arrow[d, "B^{i+1}", shift right=1] \\
\Upsilon^{i}\arrow[r, "\mathscr{D}^{i}", shift right=1]\arrow{u}{A^{i}}&\arrow{l}{\mathscr{P}^{i+1}}\Upsilon^{i+1}\arrow{u}{A^{i+1}}
\end{tikzcd}
\end{equation}
Here $F$ is an isomorphism of complexes. On the other hand, $A$ and $B$ are morphisms of complexes which induce isomorphisms in cohomology. 

Details will be explained in the rest of this section.

\subsection{Poincar\'e operators for the twisted complex}\label{sec:twisted}
We first use the Poincar\'e operators for the initial complexes to construct an isomorphism between the sum of these complexes and the twisted complex.

 \begin{lemma}\label{lem:I-PS}
We have $d(I-PS)=(I-PS)d_{V}$. 
 \end{lemma}
 \begin{proof}
We have
 $$
(I-PS)d_{V}= (I-PS)(d-S)=d-S+PdS=d-dPS=d(I-PS),
 $$
 where we have used $dP+Pd=I$, $S^{2}=0$ and  $dS=-Sd$. This implies the desired result.
 \end{proof}

Since the bounded operator $PS$ maps each of the spaces $Z^{i,j}$ to $Z^{i,j-1}$ it is evidently nilpotent and hence $I-PS$ is invertible. Define $F^{i}:=(I-P^{i+1}S^{i})^{-1}=\sum_{l=0}^{\infty}(P^{i+1}S^{i})^{l}$, noting that the sum is finite since $PS$ is nilpotent. Lemma \ref{lem:I-PS} implies that the following diagram commutes:
\begin{equation} \label{KddVK}
\begin{tikzcd}
\cdots \arrow{r}{} &Y^{i}\arrow{d}{F^{i}} \arrow{r}{d} &Y^{i+1} \arrow{d}{F^{i+1}}\arrow{r}{} & \cdots\\
\cdots \arrow{r}{} &Y^{i}\arrow{r}{d_{V}} &Y^{i+1}\arrow{r}{} & \cdots.
 \end{tikzcd}
\end{equation}

 From the commuting diagram \eqref{KddVK}, we immediately get bounded Poincar\'e operators for the twisted complex. 
  \begin{theorem}\label{thm:twisted}
Define 
\begin{equation}\label{def:PV}
P_{V}^{i}:=F^{i-1}P^{i}(F^{i})^{-1}=\sum_{l=0}^{\infty}(P^{i}S^{i-1})^{l} P^{i}(I-P^{i+1}S^{i}).
\end{equation}
We have 
\begin{equation}\label{poincare-twisted}
 d^{i-1}_{V}P^i_{V}+P^{i+1}_{V}d^i_{V}=I \text{\ for\ }i>0\text{\ and\ } d^0_VP^1_Vd^0_V=d^0_V.
\end{equation} 
\end{theorem}
\begin{proof}
 Exert $F$ on the left and $F^{-1}$ on the right of \eqref{DPPD-input} and use Lemma \ref{lem:I-PS}.
\end{proof}

 In the special case where the diagram has two rows ($N=2$), $F^{i}$ has the form
$$
F^{i}=
\left (
\begin{array}{cc}
I & P^{i+1}S^{i}\\
0 & I
\end{array}
\right ).
$$ 
 
Now the definition \eqref{def:PV} of $P_{V}$ boils down to $P_{V}=(I+PS)P(I-P S )$. If $P\circ P=0$, we  further have $P_{V}=P+PSP$, which has the matrix form
\begin{equation}\label{PV}
P_{V}=\left ( 
\begin{array}{cc}
I & PS\\
0 & I
\end{array}
\right )\left ( 
\begin{array}{cc}
P & 0 \\
0 & P
\end{array}
\right )\left ( 
\begin{array}{cc}
I & -PS\\
0 & I
\end{array}
\right )=\left ( 
\begin{array}{cc}
P & -P(PS-SP)\\
0 & P
\end{array}
\right ).
\end{equation}
 

\subsection{BGG complexes}\label{sec:bgg}

For this step, we have to make an additional assumption, namely that for each $i,j$, the operator $S^{i,j}:Z^{i,j}\to Z^{i+1,j-1}$ has closed range. This of course implies that also $S^i:Y^i\to Y^{i+1}$ has closed range. The spaces showing up in the BGG complex associated to the twisted complex are then closed linear subspaces of the spaces $Y^i$ that form the twisted complex. One defines $\Upsilon^{i}:=\ker(S^{i})\cap \ran(S^{i-1})^{\perp}$, or more explicitly, 
$$
\Upsilon^{i}:=\{(u_{0}, \cdots u_{N})\in Y^{i}: u_{0}\in \ran(S^{i-1, 1})^{\perp}, ~u_{\ell}\in \ran(S^{i-1, \ell+1})^{\perp}\cap \ker(S^{i, \ell}), \ell=2, \cdots, N-1, ~ u_{N}\in \ker(S^{i, N})\}.
$$
Intuitively, $\Upsilon^{\bs}$ comes from eliminating (as much as possible) subspaces that are connected isomorphically by the $S$ operators from \eqref{diag:multi-rows}. One then defines bounded operators
 \begin{equation}\label{reduced-complex}
\begin{tikzcd}
 \cdots \arrow{r}{ }&\Upsilon^{i-1} \arrow{r}{\mathscr{D}^{i-1}}  &\Upsilon^{i} \arrow{r}{\mathscr{D}^{i}}   &\Upsilon^{i+1}\to \cdots
 \end{tikzcd}
 \end{equation}
 as follows. Since $S^i\circ S^{i-1}=0$, we conclude that $\ran(S^{i-1})$ is a closed linear subspace of $\ker(S^i)$ and indeed $\ker(S^i)=\ran(S^{i-1})\oplus\Upsilon^i$. Let us denote by $\Pi_{\ker(S^{i})}$, $\Pi_{\ran(S^{i})^{\perp}}$ and $\Pi_{\Upsilon^{i}}$ be the (orthogonal) projections onto $\ker(S^{i})$, $\ran(S^{i})^{\perp}$ and $\Upsilon^{i}$, respectively. For simplicity of notation, we will simply write $\Pi_{\ker}$, $\Pi_{\ran^{\perp}}$ and $\Pi_{\Upsilon}$, when there is no danger of confusion.
Recall from \cite{vcap2022bgg} that we define $T:=(S|_{\ker(S)^\perp})^{-1}\Pi_{\ran(S)}$ and denote its components by $T^{i, j}: Z^{i, j}\to Z^{i-1, j+1}$, so $T^{i, j}:=(S^{i-1, j+1}|_{\ker(S)^\perp})^{-1}\circ \Pi_{\ran(S^{i-1, j+1})}$. Then $G^{i}: Z^{i}\to Z^{i-1}$ is defined by $G^{i}:=-\sum_{k=0}^{\infty}(T^{i}d^{i-1})^{k}T^{i}$. Next, one defines $A^i:\Upsilon^i\to Y^i$ by $A^{i}:=I-G^{i+1}d_{V}^{i}$ and the operators $\mathscr{D}^{\bs}$ in \eqref{reduced-complex} are given by $\mathscr{D}^{i}:=P_{\Upsilon^{i+1}}d_{V}^{i}A^{i}$. It then turns out that $d^i_VA^i=A^{i+1}\mathscr{D}^{i}$ so $A$ defines a cochain map from the sequence \eqref{reduced-complex} to the twisted complex. Via this, the fact that $d^{i+1}_V\circ d^i_V=0$ easily implies that also \eqref{reduced-complex} is a complex, and hence $A$ induces a map in cohomology. Finally,  $B^{i}: Y^{i}\to \Upsilon^{i}$ defined by $B^{i}:=\Pi_{\Upsilon^{i}}(I-d_{V}^{i-1}G^{i})$ provides a cochain map in the other direction, such that $BA=I$. One also shows that $AB$ is chain homotopic to the identity, whence these two chain maps induced isomorphisms in cohomology, which are inverse to each other. 

Recalling the diagram \eqref{diagram-two-steps}, we define the composition of chain maps $\pi^{i}:=B^{i} F^{i}$ and $\pi_{\dagger}^{i}:=(F^{i})^{-1} A^{i}$. Now we define Poincar\'e operators for the BGG complex by 
\begin{align}\label{def:P}
\mathscr{P}^{i}:= B^{i-1}P^{i}_{V} A^{i}= \pi^{i-1} P^{i} \pi_{\dagger}^{i}.
\end{align}
\begin{theorem}
The Poincar\'e operators satisfy 
$$
\mathscr{D}^{i-1}\mathscr{P}^{i}+\mathscr{P}^{i+1}\mathscr{D}^{i}=I { \text{\ for\ }i>0\text{\ and\ }\mathscr{D}^0\mathscr{P}^1\mathscr{D}^0=\mathscr{D}^0.}
$$
\end{theorem}
\begin{proof}
By definition, we have
$$
\mathscr{D}\mathscr{P}+\mathscr{P}\mathscr{D}=\mathscr{D}BP_{V}A+BP_{V}A\mathscr{D}.
$$
Now $A^{\bs}$ and $B^{\bs}$ are cochain maps, meaning that $\mathscr{D}B=Bd_{V}$ and $A\mathscr{D}=d_{V}A$. Moreover, $BA=I$ \cite[p.9]{vcap2022bgg}.  This immediately implies the claim.
\end{proof}

In the case where the diagram has two rows ($N=2$),
\begin{equation}\label{ABN2}
A=\left (
\begin{array}{cc}
I & 0\\
Td & \Pi_{\ker}
\end{array}
\right ), \quad 
B=\left ( 
\begin{array}{cc}
\Pi_{\ran^{\perp}} & 0\\
\Pi_{\ker}dT & \Pi_{\ker}
\end{array}
\right ),
\end{equation}
and
\begin{align*}
\mathscr{P}=BP_{V}A&=\left ( 
\begin{array}{cc}
\Pi_{\ran^{\perp}} & 0\\
\Pi_{\ker}dT & \Pi_{\ker}
\end{array}
\right )\left ( 
\begin{array}{cc}
P & -P(PS-SP)\\
0 & P
\end{array}
\right )
\left (
\begin{array}{cc}
I & 0\\
Td & \Pi_{\ker}
\end{array}
\right )\\
&=\left (
\begin{array}{cc}
\Pi_{\ran^{\perp}}(P-P(PS-SP)Td) & -\Pi_{\ran^{\perp}}P(PS-SP)\Pi_{\ker}\\
\Pi_{\ker}dT(P-P(PS-SP)Td)+\Pi_{\ker}PTd & -\Pi_{\ker}dTP(PS-SP)\Pi_{\ker}+\Pi_{\ker}P\Pi_{\ker}
\end{array}
\right ).
\end{align*}
 The explicit formula has several terms, but in examples, many of the $S$ operators are injective or surjective. This will simplify the formulas. More precisely, if $S^{i, j}: Z^{i, j}\to Z^{i+1, j-1}$ is injective, then $Z^{i, j}$ is removed from the BGG complex $(\Upsilon^{\bs}, \mathscr{D}^{\bs})$, i.e., the $i$-th component of $\Upsilon^{j}$ is zero. Similarly, if $S^{i, j}: Z^{i, j}\to Z^{i+1, j-1}$ is surjective, then $Z^{i+1, j-1}$ is removed.
 

\section{Examples}\label{sec:examples}
In this section, we present the construction of Poincar\'e operators for the examples derived from representation theory in \cite{vcap2022bgg}. In particular, we verify that the assumptions in the previous section are satisfied. Moreover, in this setup, we discuss additional properties such as the polynomial-preservation and the complex property.



\subsection{The setting of \cite{vcap2022bgg}}\label{sec:repr-bgg}
Let us first describe explicitly how to apply the general scheme derived in Section \ref{sec:construction} to the BGG complexes constructed {using representation theory} in \cite{vcap2022bgg}. We will also deduce that in this setting the resulting operators preserve appropriately chosen subspaces of polynomial maps which is important in applications to finite element methods. The basic input for the construction in  \cite{vcap2022bgg} (in dimension $n$) is the choice of a representation $\mathbb V$ of $O(n)$ which decomposes as $\mathbb V=\oplus_{i=0}^N\mathbb V_i$ into $O(n)$-invariant subspaces. {For appropriate choices of $\mathbb V$,} representation theory provides us with linear maps $\partial^{i.j}:\Lambda^i\mathbb R^{n*}\otimes\mathbb V_j\to\Lambda^{i+1}\mathbb R^{n*}\otimes\mathbb V_{j-1}$ such that $\partial^{i+1,j-1}\circ\partial^{i.j}=0$. Then one fixes a number $q\in\mathbb R$ and a bounded domain $\Omega\subset\mathbb R^n$ and defines $Z^{i,j}:=H^{q-i-j}(\Omega,\Lambda^i\mathbb R^{n*}\otimes\mathbb V_j)$, which can be viewed as the Sobolev space of differential $i$-forms on $\Omega$ with values in the vector space $\mathbb V_j$ and Sobolev regularity $q-i-j$. In particular, for each $i,j$, the exterior derivative defines a bounded linear operator $d^{i,j}:Z^{i,j}\to Z^{i+1,j}$, because the Sobolev regularity in the target is one lower than in the source. On the other hand, the Sobolev regularities in $Z^{i,j}$ and $Z^{i+1,j-1}$ are the same, so $\partial^{i.j}$ induces a bounded linear operator $S^{i,j}:Z^{i,j}\to Z^{i+1,j-1}$ with closed range. It is shown in \cite{vcap2022bgg} that all the assumptions we made in Section \ref{sec:construction} are satisfied and that the cohomology of $(Y^\bs,d^\bs_V)$ is isomorphic to the direct sum of the cohomologies of the complexes $(Z^{\bs,j},d^{\bs,j})$.

In this situation, appropriate Poincar\'e operators for the de Rham complex are constructed in the work \cite{costabel2010bogovskiui} of Costabel and McIntosh. For a domain $\Omega\subset\mathbb R^n$ which is starlike with respect to all points in some open ball, they construct integral operators $R_\ell$ on compactly supported $\ell$-forms on $\Bbb R^n$ which satisfy appropriate null homotopy relations and by Theorem 3.2 of that reference, these are pseudodifferential operators of order $-1$ with symbol in an appropriate H\"ormander class. In particular, for bounded domain which is starlike to some ball, they induce bounded linear operators $H^s(\Omega,\Lambda^\ell\mathbb R^{n*})\to H^{s+1}(\Omega,\Lambda^{\ell-1}\mathbb R^{n*})$ that satisfy appropriate homotopy relations. Now of course we can also identify $Z^{i,j}$ with $H^{q-i-j}(\Omega,\Lambda^i\mathbb R^{n*})\otimes\mathbb V_j$, so tensorizing with the identity on $\mathbb V_j$, the operator $R_i$ gives rise to bounded linear operators $P^{i,j}:Z^{i,j}\to Z^{i-1,j}$ for each $i$ and $j$. Hence the construction in Section \ref{sec:construction} applies and we get

\begin{theorem}\label{thm:bounded-poincare}
Let $\Omega\subset\Bbb R^n$ be a bounded domain which is starlike with respect to an open ball. Then for any representation $\mathbb V$ as discussed in Section 3 of \cite{vcap2022bgg} there is a bounded Poincar\'e operator for both the twisted complex and the BGG complex on $\Omega$ constructed from $\Bbb V$ as in that reference.   
\end{theorem}

For later use, we observe that the Poincar\'e operators for the twisted complex have additional analytic properties. Indeed, since the Poincar\'e operators for the de Rham complex we start with are restrictions of pseudodifferential operators of order $-1$ on $\Bbb R^n$, the same is true for the operators $PS$ and hence its powers are such restrictions, too. Consequently, also the Poincar\'e operators for the twisted complex obtained in Theorem \ref{thm:bounded-poincare} are restrictions of pseudodifferential operators of order $-1$. They also have the same support properties as the Poincar\'e operators for the de Rham complex constructed in \cite{costabel2010bogovskiui}. 

\subsection{Polynomial-preservation}\label{sec:Polynomial-preservation}
A remarkable property of the constructions in \cite{costabel2010bogovskiui} is that the operators preserve subspaces of differential forms with polynomial coefficients under rather weak and general assumptions. We can carry this over to the BGG complexes constructed in \cite{vcap2022bgg} under appropriate assumptions. Following \cite{costabel2010bogovskiui}, we take a kind of axiomatic approach to this, i.e., we consider subspaces {$\Pi^{i,j}\subset C^\infty(\Bbb R^n,\Lambda^i\Bbb R^{n*}\otimes\Bbb V_j)$ of polynomial maps} and impose conditions on these subspaces that are needed in order to make sure that they are preserved by the Poincar\'e operator. The first two conditions are the exact analog of the conditions in  \cite{costabel2010bogovskiui}:
\begin{itemize}
\item[(a)] For any $t\in\Bbb R$, $a\in\Bbb R^n$ and $u\in\Pi^{i,j}$, also $x\mapsto u(tx+a)$ lies in $\Pi^{i,j}$.
  \item[(b)] For any $u\in\Pi^{i,j}$ and the Euler vector field $E(x)=x$, we get $i_Eu\in \Pi^{i-1,j}$.   
\end{itemize}
The following condition, which would make sense in the setting of \cite{costabel2010bogovskiui} is not explicitly stated or used there. It is very natural, however, for constructing subcomplexes of polynomial forms and it is needed to pass to BGG complexes. 
\begin{itemize}
\item[(c)] For any {$i, j$}, $d(\Pi^{i,j})\subset \Pi^{i+1,j}$.  
\end{itemize}
Finally, we need two conditions relating the subspace to the operators that occur in the BGG construction, namely: 
\begin{itemize}
\item[(d)] For any $i<n$ and $j>0$ we have $S(\Pi^{i,j})\subset \Pi^{i+1,j-1}$.
\item[(e)] For any $i>0$ and $j<N$ we have $T(\Pi^{i,j})\subset \Pi^{i-1,j+1}$.   
\end{itemize}
Now on any bounded domain $\Omega\subset\Bbb R^n$, the restriction of any form with polynomial coefficients lie in any Sobolev space. Consequently,  the spaces $\Pi^{i,j}$ are always contained in the spaces $Z^{i,j}$ used above  and we get
\begin{corollary}\label{cor:polynomial}
 In the setting of Theorem \ref{thm:bounded-poincare} above, suppose that we have selected spaces $\Pi^{i,j}$ of polynomial $i$-forms with coefficients in $\Bbb V_j$ such that the assumptions (a)--(e) from above are all satisfied and put $\Pi^i:=\oplus_j \Pi^{i,j}$. Then the Poincar\'e operators for the twisted complex constructed in Theorem \ref{thm:bounded-poincare} map each $\Pi^i$ to $\Pi^{i-1}$, while the Poincar\'e operators for the BGG complex map each $\Upsilon^i\cap \Pi^i$ to $\Upsilon^{i-1}\cap \Pi^{i-1}$. 
\end{corollary}
\begin{proof}
  By  \cite{costabel2010bogovskiui}, assumptions (a) and (b) imply that the Poincar\'e operator $P$ for the de Rham complex have the property that $P^{i,j}(\Pi^{i,j})\subset\Pi^{i-1,J}$ and together with (c), we conclude that $(I-P^{i+1}S^i)(\Pi^i)\subset\Pi^i$. Similarly, all powers of $P^iS^{i-1}$ map $\Pi^{i-1}$ to itself, so the claim on the twisted complex follows readily from Theorem \ref{thm:twisted}.
  
  To pass to the BGG complex, we just have to observe that by definition, we get $G^i(\Pi^i)\subset\Pi^{i-1}$ and $A^i(\Pi^i)\subset\Pi^i$. For the map $B^i$, we just have to observe that the projection onto $\Upsilon^i$ can be written as $I-S^{i-1}T^i-T^{i+1}S^i$, so this also maps $\Pi^i$ to itself. Hence also the claim on the BGG complex follows directly from the construction. 
\end{proof}

Observe that the conditions (a)--(e) are not difficult to satisfy. For example, we can fix some homogeneity $m>n+N$, denote by $\mathcal H^r$ the space of homogeneous polynomials of degree $r\in\Bbb N$ and then put $\Pi^{i,j}:=\mathcal H^{m-i-j}\otimes\Lambda^i\Bbb R^{n*}\otimes\Bbb V_j$. The conclusion also holds for {$ \Pi^{i,j}:=\mathcal P_{m-i-j}\otimes\Lambda^i\Bbb R^{n*}\otimes\Bbb V_j$, where
  $\mathcal P_r=(\oplus_{\ell\leq r}\mathcal H^r)$ is the space of polynomials of degree at most $r$.}

\subsection{Complex property}\label{sec:complex-prop}
Another favorable property of a Poincar\'e operator is the complex property, i.e.\ that $P^{i-1}\circ P^i=0$ for any $i$. This property is obviously preserved under the passage to the twisted complex discussed in Section \ref{sec:twisted}. The situation is not so simple for the passage to the BGG complex. From the definition, it follows directly that
$$
\mathscr{P}^{i-1}\mathscr{P}^{i}=B^{i-2}P^{i-1}_{V} A^{i-1}B^{i-1}P^{i}_{V} A^{i},
$$
and $A^{i-1}B^{i-1}$ is far from being the identity. Fortunately, there is a universal solution to this problem, which does not destroy boundedness or the property of being a pseudodifferential operator.

   \begin{theorem}\label{thm:complex}
     Let $(V^{\bs}, D^{\bs})$ be a complex and $P^{\bs}$ be a Poincar\'e operator satisfying the null homotopy relation $D^{i-1}P^{i}+P^{i+1}D^{i}=I$. Define $\tilde{  P^i}:=  P^i-   {D^{i-2}} P^{i-1}P^i$. Then $\tilde P^\bs$ satisfies the null homotopy relation  $D^{i-1}\tilde{P}^i+\tilde{P}^{i+1} D^i=I$ and the complex property $\tilde P^{i-1}\circ \tilde P^i=0$.

     If the spaces $V^i$ are Banach spaces and the operators $D^i$ and $P^i$ all are bounded, then also the modified operators $\tilde P^i$ are bounded. 
   \end{theorem} 
   \begin{proof}
     The last statement is obvious from the definitions. For the relations, we compute using $D^2=0$ 
{ 
$$
  D\tilde P+\tilde PD = DP+PD-DPPD.
$$}
Using the null homotopy relation twice, we get
{ 
$$
 DPPD=-PDPD+PD=PDDP=0,
 $$}
 so $\tilde P^\bs$ satisfies the null homotopy relation, too. Finally, 
$$
 \tilde P^2= P^{2}- D P^{3}- P D P^{2}+ D P^{2}D P^{2}= P^{2}-( D P+ P D) P^{2}+
 D P^{2} D P^{2}= D P^{2} D P^{2}, 
$$
and this vanishes since $ D P^{2} D=0$.
\end{proof}
{ 
\begin{remark}
By the null-homotopy relation, we observe that $\tilde{  P^i}:=  P^i-{D^{i-2}} P^{i-1}P^i={P^{i}} D^{i-1}P^i$.
\end{remark}
}


\section{The case of nontrivial cohomology}\label{sec:modifications}
We next discuss modifications that are needed to treat cases with non-trivial cohomology, so the null homotopy relation \eqref{DPPD-general} gets replaced by the more general version \eqref{DPPD-L}. While the basic ideas are not so different from the case of vanishing cohomology, there is the question of which properties one should assume about $L$ and how to prove that they are preserved by the construction. This is particularly relevant for the examples related to de Rham complexes where the form of these properties is rather unexpected. Consequently, we will discuss two different situations here. One basic case that was mentioned in Section \ref{sec:de Rham} already, is that the image of $L$ coincides with the cohomology, or equivalently that $dL=Ld=0$. We can deal with this case in a general abstract setting. The second case is specifically tailored to the examples coming from de Rham complexes and representation theory as discussed in Section \ref{sec:repr-bgg}. Here the treatment is based on a variant of the construction for the de Rham complex in \cite{costabel2010bogovskiui}.

\subsection{$L$ operators representing cohomology}
We discuss the first case, where we require that $dL=0$ and hence $\ran(L)$ coincides with the cohomology. Our goal is to construct the corresponding operators $P_{V}$, $L_{V}$ for the twisted complex and $\mathscr{P}$, $\mathscr{L}$ for the BGG complex, such that analogous properties hold. To this end, we will follow the derivation of the twisted and BGG complexes, and carry over properties of the input operators $P$, $L$ to these cases by complex maps $F$ and $A$, $B$. The main difficulty is that in the setup in Section \ref{sec:construction}, the complex maps $F^{\bs}$ from the input to the twisted complex involves Poincar\'e operators $P^{\bs}$ of the input complexes and the null-homotopy identity $dP+Pd=I$ plays an important role in verifying the properties. {Hence in the case of nontrivial cohomology, a different argument is needed.

One option is to 
}
return to the original setting of \cite{vcap2022bgg} and make different assumptions for the two steps. For the first step, we also need the $K$-operators from the first step there. So in addition to the assumptions in the beginning of Section  \ref{sec:construction}, we assume that we have given bounded operators $K^{i,j}:Z^{i,j}\to Z^{i,j-1}$ such that $S^{i,j}=d^{i,j-1}K^{i,j}-K^{i+1,j}d^{i,j}$ and such that $S^{i,j-1}K^{i,j}=K^{i+1,j-1}S^{i,j}$. Then it was shown in Theorem 1 of \cite{vcap2022bgg} that viewing our operators as $K^i:Y^i\to Y^i$, the exponentials $F^i=\exp(K^i)$ (which are just polynomials, since each $K^i$ is nilpotent) are bounded operators that fit together to define an isomorphism $F^\bs:(Y^\bs,d^\bs)\to (Y^\bs,d_V^\bs)$ of differential complexes. So in particular, we have $F^{i+1}\circ d^i=d^i_V\circ F^i$. Using this, the passage to the twisted complex is fairly trivial. Given a Poincar\'e operator $P^\bs$ for $(Y^\bs,d^\bs)$ we define $\tilde P^i_V:Y^i\to Y^{i-1}$ by $\tilde P^i_V= F^i\circ P^i\circ (F^{i-1})^{-1}$.

\begin{proposition}\label{prop:sharp-cohom-twisted}
Suppose that $P^\bs$ is a Poincar\'e operator for $(Y^\bs,d^\bs)$ such that $dP+Pd=I-L$. Then $\tilde P^\bs$ is a Poincar\'e operator for $(Y^\bs,d_V^\bs)$ with $d_V\tilde P_V+\tilde P_Vd_V=I-L_V$, where $L^i_V=F^iL^i(F^i)^{-1}$. If each $P^i$ is bounded then each $\tilde P_V^i$ is bounded and if $P^\bs$ has the complex property then $\tilde P^\bs$ has the complex property, too. Finally, if $L^i$ has finite rank then the same is true for $L^i_V$ and if $dL=Ld=0$, then also $d_VL_V=L_Vd_V=0$.  
\end{proposition}
\begin{proof}
Since $Fd=dF$, we immediately get $d_V\tilde P_V=FdPF^{-1}$ and $\tilde P_Vd_V=FPdF^{-1}$ which immediately implies that homotopy relation with the right form of $L^i_V$. The remaining statements then are obvious, they just carry over through conjugation by $F$. 
\end{proof}

For the passage to the BGG complex, we don't need the $K$-operators any more. So we are back to the setting from the beginning of Section \ref{sec:construction} and in addition, we assume that we have given a Poincar\'e operator $P_V^\bs$ for the twisted complex $(Y^\bs,d_V^\bs)$. As in Section \ref{sec:bgg} we define the operators $\mathscr{P}^{i}:\Upsilon^i\to\Upsilon^{i-1}$ as $\mathscr{P}^{i}:=B^{i-1}P^{i}_{V} A^{i}$.
\begin{theorem}
Suppose that the Poincar\'e operator $P_V^\bs$ satisfies a homotopy relation $d_VP_V+P_Vd_V=I-L_V$ such that $d_VL_V=L_Vd_V=0$. Then the induced Poincar\'e operator $\mathscr{P}^\bs$ on the BGG complex satisfies a homotopy relation of the form $\mathscr{D}\mathscr{P}+\mathscr{P}\mathscr{D}=I-\mathscr{L}$, with $\mathscr{L}\mathscr{D}=\mathscr{D}\mathscr{L}=0$. Moreover, if $L_V^2=L_V$ {then} also $\mathscr{L}^2=\mathscr{L}$. Finally, if $P_V$ has the complex property then $L_V^2=L_V$ and we can modify $\mathscr{P}^\bs$ in such a way that in addition to all proceeding properties, it has the complex property, too. 
\end{theorem}
\begin{proof}
  As before, we see (without assumptions on $L_V$) that $\mathscr{P}^\bs$ satisfies the homotopy relation with $\mathscr{L}^i=B^iL_V^iA^i$, so the question is whether the nice properties are preserved. But then $A^i\mathscr{D}^{i-1}=d^i_VA^i$ and $\mathscr{D}^{i-1}B^{i-1}=B^id_V^{i-1}$ together with $d_VL_V=L_Vd_V=0$ implies that $\mathscr{L}\mathscr{D}=\mathscr{D}\mathscr{L}=0$. By construction, we also get $(\mathscr{L}^i)^2=B^iL^i_VA^iB^iL^i_VA^i$. Next, equation (25) of \cite{vcap2022bgg} shows that $A^iB^i=I-d^{i-1}_VG^i-G^{i+1}d^i_V$ and hence we conclude that $L^i_VA^iB^iL^i_V=(L^i_V)^2=L^i_V$ and hence $(\mathscr{L}^i)^2=\mathscr{L}^i$. 

  To prove the last statement, assume that $P^{i-1}_VP^i_V=0$. Then the homotopy relation readily implies that $L_VP_V=P_V-P_Vd_VP_V=P_VL_V$. Using this and $L_Vd_V=d_VL_V=0$, we see that composing the homotopy relation with $L_V$ gives $0=L_V-(L_V)^2$, so $(L_V)^2=L_V$ follows. Now we first define $\widehat{\mathscr{P}}^i:=\mathscr{P}^i-\mathscr{P}^i\mathscr{L}^i-\mathscr{L}^{i-1}\mathscr{P}^i+\mathscr{L}^{i-1}\mathscr{P}^i\mathscr{L}^i$. This implies that $\mathscr{D}\widehat{\mathscr{P}}=\mathscr{D}\mathscr{P}$ and $\widehat{\mathscr{P}}\mathscr{D}=\mathscr{P}\mathscr{D}$, so the homotopy relation holds for $\widehat{\mathscr{P}}$ in the same form as for $\mathscr{P}$. But then $\mathscr{L}^2=\mathscr{L}$ and the definition of $\widehat{\mathscr{P}}$ readily imply that $\mathscr{L}\widehat{\mathscr{P}}=\widehat{\mathscr{P}}\mathscr{L}=0$.

  This brings us into a situation in which the trick used in the proof of Theorem \ref{thm:complex} works. The homotopy relation for $\widehat{\mathscr{P}}$ together with $\mathscr{D}\mathscr{L}=0$ implies that $\mathscr{D}\widehat{\mathscr{P}}\mathscr{D}=\mathscr{D}$ and in turn $\mathscr{D}\widehat{\mathscr{P}}^2\mathscr{D}=0$. Then we can define $\widetilde{\mathscr{P}}^i:=\widehat{\mathscr{P}}^i-\mathscr{D}^{i-1}\widehat{\mathscr{P}}^{i-1}\widehat{\mathscr{P}}^i$ for each $i$. This definition implies that $\mathscr{D}\widetilde{\mathscr{P}}=\mathscr{D}\widehat{\mathscr{P}}$ and similar for the other composition. Thus we still have the homotopy relation with the same map $\mathscr{L}$. The proof that $\widetilde{\mathscr{P}}^2=0$ is exactly as the last computation in the proof of Theorem \ref{thm:complex}. 
\end{proof}

\begin{remark}
As we assume that there exist bounded $K$ operators, the above construction only applies to BGG complexes with smooth functions and twisted complexes with smooth and certain Sobolev functions (e.g., the rows in the input complex have equal Sobolev regularity).  This does not apply to BGG complexes with Sobolev spaces, as they will require input complexes with different kinds of regularity, for which the $K$ operators are rarely known.
\end{remark}

\subsection{Bounded Lipschitz domains}\label{sec:Lipschitz}
We now discuss the second setting for non-trivial cohomology, so we use the setup of Section \ref{sec:repr-bgg} but for the case of a bounded Lipschitz domain $\Omega$. So as there, we start with a representation $\mathbb V=\oplus_{i=0}^N\mathbb V_i$ and for a fixed $q\in\Bbb R$ define  $Z^{i,j}:=H^{q-i-j}(\Omega,\Lambda^i\mathbb R^{n*}\otimes\mathbb V_j)$. We also have the bounded operators $d^{i,j}:Z^{i,j}\to Z^{i+1,j}$ and $S^{i,j}:Z^{i,j}\to Z^{i+1,j-1}$ induced by the exterior derivative and the linear maps $\partial^{i,j}$ as there. This gives us the twisted complex $(Y^\bs,d^\bs_V)$. Following the developments in Section 4.3 of \cite{costabel2010bogovskiui}, we find a finite covering of $\overline{\Omega}$ by open sets $U_1,\dots,U_m$ such that for each $i$, the domain $\Omega_i:=U_i\cap\Omega$ is starlike with respect to some open ball. Moreover, via a partition of unity, we get smooth functions $\chi_i\in C^\infty(\Bbb R^n,\Bbb R)$ with values in $[0,1]$ such that $\supp(\chi_i)\subset U_i$ and such that $\sum_i \chi_i\equiv 1$ on some open neighborhood of $\overline{\Omega}$.

Now for each $i$, the Poincar\'e operator for the twisted complex constructed in Theorem \ref{thm:bounded-poincare} makes sense on smooth forms, so we get operators that we denote by $P^i_\ell: C^\infty(\Omega_\ell,\Lambda^i\Bbb R^{n*}\otimes\Bbb V)\to C^\infty(\Omega_\ell,\Lambda^{i-1}\Bbb R^{n*}\otimes\Bbb V)$ such that $d^{i-1}_VP^i_\ell+P^{i+1}_\ell d^i_V=I$ for each $i>0$. Hence for $u\in C^\infty(\Omega,\Lambda^i\Bbb R^{n*}\otimes\Bbb V)$ we can restrict to $\Omega_\ell$ and then consider $\chi_\ell P^i_\ell u\in C^\infty(\Omega_\ell,\Lambda^{i-1}\Bbb R^{n*}\otimes\Bbb V)$. But this vanishes identically on the open subset $\Omega\setminus(\Omega\cap \supp(\chi_i))$ which together with $\Omega_i$ covers $\Omega$, so it can be extended by zero to a form in $C^\infty(\Omega,\Lambda^{i-1}\Bbb R^{n*}\otimes\Bbb V)$. So in particular, we can define
\begin{equation}\label{tP-Lip-def}
  \tilde P^i(u):=\sum_\ell\chi_\ell P^i_\ell u\in C^\infty(\Omega,\Lambda^{i-1}\Bbb R^{n*}\otimes\Bbb V)
\end{equation}
By construction, this is agian the restriction of a pseudodifferential operator of order $-1$ and hence it defines bounded operators $\tilde P^i:Y^i\to Y^{i-1}$, where as before $Y^i=\oplus_jZ^{i,j}$. Moreover, we get 
$$
(d^i_V\tilde P^i+\tilde P^id^i_V)(u)=\textstyle\sum_\ell\chi_\ell(d^i_VP^i_\ell u+P^i_\ell d^i_Vu)+\sum_\ell(d^i_V(\chi_\ell P^i_\ell u)-\chi_\ell d^i_V(P^i_\ell u)).  
$$
The homotopy relations for the operators $P^i_\ell$ show that the first sum gives $\sum_\ell\chi_\ell u|_{\Omega_\ell}=u$ while the second sum can be written in the language of differential forms as $\sum_\ell d\chi_i\wedge P^i_\ell u$. Here we use the wedge product of a real-valued one-form and a $\Bbb V$-valued $(i-1)$-form, which is a $\Bbb V$-valued $i$-form. Now each summand in the latter sum is a form on $\Omega_\ell$ which vanishes identically on the complement of $\supp(\chi_i)$, so as above, this defines
\begin{equation}\label{tK-Lip-def}
  \tilde L^i u:=-\sum_\ell d\chi_\ell\wedge P^i_\ell u\in C^\infty(\Omega,\Lambda^i\Bbb R^{n*}\otimes\Bbb V).
\end{equation}
As above, $\tilde L^i$ is the restriction of a pseudodifferential operator of order $-1$ on $\Bbb R^n$ so again it defines a bounded operator $Y^i\to Y^i$ (which even improves Sobolev regularity). Now by definition, we  get the homotopy relation
\begin{equation}\label{tP-tK-homot}
  d^{i-1}_V\tilde P^i+\tilde P^{i+1}d_V^i=I-\tilde L^i
\end{equation}
and hitting this with $d_V$ from the left and from the right, we conclude that $d^i_V\tilde L^i=\tilde L^{i+1}d^i_V$.   

Still following \cite{costabel2010bogovskiui}, we can improve things further. This is based on the observation that the functions $\chi_\ell$ coming from a partition of unity can often be chosen to be constant on large subsets. Now suppose that $x\in\Omega$ is a point at which the family $\{\chi_\ell\}$ is \textit{flat} in the sense that each $\chi_i$ is constant on some neighborhood of $x$. Then evidently $d\chi_\ell (x)=0$ for each $\ell$ and hence $\tilde L^i(u)$ vanishes at $x$ for any $u\in  C^\infty(\Omega,\Lambda^i\Bbb R^{n*}\otimes\Bbb V)$. Now a compactness argument shows that given $\Omega$, we can find finitely many coverings $\{U^{(s)}_\ell\}$ of $\overline{\Omega}$ and associated families $\{\chi_\ell^{(s)}\}$ of functions such that each $U^{(s)}_\ell\cap\Omega$ is starlike with respect to some ball and for each $x_0\in\Bbb R^n$ at least one of the families $\{\chi_\ell^{(s)}\}$ is flat at $x_0$, see Lemma 4.5 of \cite{costabel2010bogovskiui}. Now we denote the operators for the covering $\{U^{(s)}_\ell\}$ by $\tilde P^i_{(s)}$ and $\tilde L^i_{(s)}$ for $s=1,\dots, m$ and the put
  \begin{gather}\label{P-Lip-def}
    P^i_V:=\tilde P^i_{(1)}+\tilde L^{i-1}_{(1)}\tilde P^i_{(2)}+\tilde L^{i-1}_{(1)}\tilde L^{i-1}_{(2)}\tilde P^i_{(3)}+\dots +\tilde L^{i-1}_{(1)}\dots \tilde L^{i-1}_{(m-1)}\tilde P^i_{(m)}.\\
    \label{K-Lip-def}    L^i_V:=\tilde L^i_{(1)}\dots\tilde L^i_{(m)}
  \end{gather}
  \begin{theorem}\label{thm:Lip-domains}
 (1) The operators $P^i_V:C^\infty(\Omega,\Lambda^i\Bbb R^{n*}\otimes\Bbb V)\to C^\infty(\Omega,\Lambda^{i-1}\Bbb R^{n*}\otimes\Bbb V)$ are restrictions of pseudodifferential operators of order $-1$ on $\Bbb R^n$ and  hence induce bounded operators $Y^i\to Y^{i-1}$ that we denote by the same symbol. These satisfy the relations $d^{i-1}_VP^i_V+P^{i+1}_Vd^i_V=I-L^i_V$ and we get $d^i_VL^i_V=L^{i+1}_Vd_V^i$.

(2)  The operators $L^i_V:C^\infty(\Omega,\Lambda^i\Bbb R^{n*}\otimes\Bbb V)\to C^\infty(\Omega,\Lambda^i\Bbb R^{n*}\otimes\Bbb V)$ are restrictions of pseudodifferential operators of order $-\infty$ on $\Bbb R^n$ and hence they are smoothing operators.

(3) Using the operators $P^i_V:Y^i\to Y^{i-1}$ as the starting point of the construction discussed in Section \ref{sec:bgg}, we obtain a bounded Poincare operator $\mathscr{P}^\bs$ for the BGG complex determined by $\Bbb V$ which satisfies a homotopy relation $\mathscr{D}\mathscr{P}+\mathscr{P}\mathscr{D}=I-\mathscr{L}$ where $\mathscr{L}$ is the restriction of a pseudodifferential operator of order $-\infty$ on $\Bbb R^n$ and hence a smoothing operator and invertible module compact operators. 
  \end{theorem}
  \begin{proof}
    (1) The fact that $P^i_V$ is the restriction of a pseudodifferential operator follows readily from its definition as a composition of such operators. Moreover, the homotopy relations for the operators $P^i_{(s)}$ together with $d_V^{i-1}\tilde L^i_{(s)}=\tilde L^{i+1}_{(s)}d_V^i$ recursively imply that
    $$
    d^{i-1}_V\tilde L^{i-1}_{(1)}\dots \tilde L^{i-1}_{(s)}\tilde P^i_{(s+1)}+\tilde L^i_{(1)}\dots \tilde L^i_{(s)}\tilde P^{i+1}_{(s+1)}d_V^i=\tilde L^i_{(1)}\dots \tilde L^i_{(s)}-\tilde L^i_{(1)}\dots \tilde L^i_{(s+1)}
    $$
    which immediately gives the required homotopy relation, while  $d^i_VL^i_V=L^{i+1}_Vd_V^i$ is obvious from the definition.

    (2) If the family $\{\chi_\ell^{(s)}\}$ is flat on a neighborhood of $x_0$ then for any distribution $u$, the distribution $\tilde L^i_{(s)}u$ is smooth on a neighborhood of $x_0$. But since all the other operators $\tilde L^i_{(r)}$ are pseudolocal, also the composition $L^i_V$ has the property that $L^i_Vu$ is smooth on a neighborhood of $x_0$. But by construction, this works for any $x_0\in\Bbb R^n$ and the claim follows.

    (3) We get $\mathscr{P}^i=B^{i-1}P^i_VA^i$ for the operators $A$ and $B$ from section \ref{sec:bgg}. Since $\mathscr{D}B=Bd_V$ and $A\mathscr{D}=d_VA$, we get the claimed homotopy relation with $\mathscr{L}^i=B^{i-1}L^i_VA^i$. But in our setting, the operators $A$ and $B$ both are differential operators that are defined on smooth maps on all of $\Bbb R^n$ so the representation as a composition implies that $\mathscr{L}^i$ is the restriction of a pseudodifferential operator of order $-\infty$. 
  \end{proof}

\section{Proxies and Applications}\label{sec:applications}

In this section, we express the Poincar\'e operators derived in the previous sections in terms of vector/matrix proxies. First, we discuss the case of dimension one using proxies and explicit forms of spaces and operators. The result is trivial, as one can obtain the inverse of a second-order derivative by simply integrating twice, but carrying out this simple example may illustrate the general construction. Then we discuss the case of the elasticity complex in 3D and compare the result at index one to the Ces\`aro-Volterra integral and generalizations \cite{ciarlet2010cesaro}. We also present explicit results for the elasticity version of the twisted complex, inspired by their applications in continuum mechanics. 

 Finally, we present an application of the newly derived Poincar\'e operators on a systematic construction of polynomial versions of BGG complexes. The homotopy relation and the polynomial-preserving property imply their exactness.


\subsection{Poincar\'e operators in one dimension.} The geometric version of the BGG machinery usually is not discussed in dimension one, since the corresponding geometric structures are only available in higher dimensions. Still one can follow the representation theory setup discussed in Section 3 of \cite{vcap2022bgg} for the Lie algebra $\mathfrak g=\mathfrak{sl}(2,\mathbb R)$ with $\mathfrak g_0=\Bbb R$ and use the corresponding maps to setup BGG diagrams. The simplest case corresponds to the standard representation $\Bbb R^2$ of $\mathfrak g$. This leads to the following BGG diagram in one space dimension for any real number $q$.
\begin{equation} \label{diagram-1D}
\begin{tikzcd}
 H^{q} \arrow{r}{\partial} &H^{q-1} \\
 H^{q-1}\arrow{r}{\partial} \arrow[ur, "I"]&H^{q-2},
 \end{tikzcd}
\end{equation}
where $\partial:=\frac{d}{dx}$.
The BGG complex is obtained by simply connecting the two rows:
\begin{equation}\label{BGG-1D}
\begin{tikzcd}
H^{q}\arrow{r}{\partial^{2}}& H^{q-2}.
\end{tikzcd}
\end{equation}

Suppose that for the two rows of \eqref{diagram-1D} we have Poincar\'e operators $P_{\sharp}: H^{q-1}\to H^{q}$ and $P_{\flat}: H^{q-2}\to H^{q-1}$, respectively. That is, $\partial P_{\sharp}=I_{H^{q-1}}$, $P_{\sharp}\partial =I_{H^{q}}+c$, where $c$ is a constant. Similar identities hold for $P_{\flat}$. Then we want to derive Poincar\'e operators for the BGG complex \eqref{BGG-1D} satisfying $\partial^{2}\mathscr{P}=I_{H^{q-2}}$ and $\mathscr{P}\partial^{2}=I_{H^{q}}+g$, where $g\in \ker(\partial^{2})$ is a linear function.
In fact,   one can verify that composition of the de~Rham Poincar\'e operators $\mathscr{P}:=P_{\sharp}\circ P_{\flat}$ gives the desired properties. Nevertheless, we go through our construction to illustrate the theory. 

The first step is to derive the Poincar\'e operators for the twisted complex by the following diagram:
  \begin{equation} \label{twisted}
\begin{tikzcd}[ampersand replacement=\&, column sep=1in]
\left (
\begin{array}{c}
H^{q}\\
H^{q-1} 
\end{array}
\right )\arrow{d}{F^{0}} \arrow{r}{\left (
\begin{array}{cc}
\partial &\\
&\partial
\end{array}
\right ) } \&\left (
\begin{array}{c}
H^{q-1} \\
H^{q-2} 
\end{array}
\right )\arrow{d}{F^{1}}  
\\
\left (
\begin{array}{c}
H^{q}\\
H^{q-1} 
\end{array}
\right ) \arrow{r}{\left (
\begin{array}{cc}
\partial & -I\\
0&\partial
\end{array}
\right ) } \&\left (
\begin{array}{c}
H^{q-1} \\
H^{q-2} 
\end{array}
\right ),
 \end{tikzcd}
\end{equation}
where
$$F^{0}:=
\left (
\begin{array}{cc}
I & P_{\sharp}\\
0 & I
\end{array}
\right ), \quad 
F^{1}:=
\left (
\begin{array}{cc}
I & 0\\
0 & I
\end{array}
\right ).
$$
Using the commuting diagram, we define 
$$
P_{V}:=F^{0}\circ P\circ (F^{1})^{-1}=\left (
\begin{array}{cc}
I & P_{\sharp}\\
0 & I
\end{array}
\right )\left (
\begin{array}{cc}
P_{\sharp} & 0\\
0 & P_{\flat}
\end{array}
\right )\left (
\begin{array}{cc}
I & 0\\
0 & I
\end{array}
\right )=\left (
\begin{array}{cc}
P_{\sharp} & P_{\sharp}P_{\flat}\\
0 & P_{\flat}
\end{array}
\right ).
$$

The operators in \eqref{ABN2} now boil down to
$$A^{0}:=
\left (
\begin{array}{cc}
I & 0\\
\partial & 0
\end{array}
\right ), \quad 
A^{1}:=
\left (
\begin{array}{cc}
I & 0\\
0 & I
\end{array}
\right )\quad 
B^{0}:=
\left (
\begin{array}{cc}
I & 0\\
0 & 0
\end{array}
\right ), \quad 
B^{1}:=
\left (
\begin{array}{cc}
0 & 0\\
\partial & I
\end{array}
\right ),
$$
This leads to the following commuting diagram:
\begin{equation}\label{BGG-1D-iso}
\begin{tikzcd}[ampersand replacement=\&, column sep=1in]
\left (
\begin{array}{c}
H^{q}\\
H^{q-1} 
\end{array}
\right )\arrow{d}{B^{0}} \arrow{r}{\left (
\begin{array}{cc}
\partial &-I\\
0&\partial
\end{array}
\right ) } \&\left (
\begin{array}{c}
H^{q-1} \\
H^{q-2} 
\end{array}
\right )\arrow{d}{B^{1}}  
\\
\left (
\begin{array}{c}
H^{q}\\
0
\end{array}
\right ) \arrow{r}{\left (
\begin{array}{cc}
0 & 0\\
\partial^{2} & 0
\end{array}
\right ) } \&\left (
\begin{array}{c}
0 \\
H^{q-2} 
\end{array}
\right ),
 \end{tikzcd}
\end{equation}

Then we obtain Poincar\'e operators for the bottom row of \eqref{BGG-1D-iso}:
$$
 {\mathscr{P}}=B^{0}\circ P_{V} \circ A^{1}=\left (
\begin{array}{cc}
I & 0\\
0 & 0
\end{array}
\right )\left (
\begin{array}{cc}
P_{\sharp} & P_{\sharp}P_{\flat}\\
0 & P_{\flat}
\end{array}
\right )\left (
\begin{array}{cc}
I & 0\\
0 & I
\end{array}
\right )=\left (
\begin{array}{cc}
P_{\sharp} & P_{\sharp}P_{\flat}\\
0 & 0
\end{array}
\right ).
$$
Note that the first component of $\Upsilon^{1}$ is zero because the $S$ operator is bijective ($\ran(S)^{\perp}=0$). Therefore
$$
\mathscr {P}\left (
\begin{array}{c}
0 \\
u 
\end{array}
\right )=\left (
{ \begin{array}{c}
    P_{\sharp}P_{\flat}u\\
    0
\end{array}}
\right )
$$
gives the Poincar\'e operator $\mathscr{P}:=P_{\sharp}P_{\flat}$ for the nontrivial component.

\subsection{Vector and matrix proxies}
In higher dimensions, the diagrams are more delicate and more structures come in. Following \cite{arnold2021complexes,vcap2022bgg}, we summarize our notations for vector/matrix proxies in the table above.
\begin{table}[h!]
\begin{center}
\begin{tabular}{c|c}
$\mathbb V$ & $\mathbb R^n$  \\
$\mathbb M$ &the space of all $n\times n$-matrices\\
$\mathbb S$ & symmetric matrices\\
$\mathbb K$ & skew symmetric matrices\\
$\mathbb T$ & trace-free matrices\\
$\skw: \M\to \K$ & skew symmetric part of a matrix\\
$\sym: \M\to \S$ & symmetric part of a matrix\\
$\tr:\M\to\R$ & matrix trace\\
$\iota: \R\to \M$  & the map $\iota u:= uI$ identifying a scalar with a scalar matrix\\
$\dev:\mathbb{M}\to \mathbb{T}$ & deviator (trace-free part of a matrix) given by $\dev u:=u-1/n \tr (u)I$\\
\end{tabular}
\end{center}
\end{table}

Moreover, in three space dimensions, we can identify a skew symmetric matrix with a vector,
$$
 \mskw\left ( 
\begin{array}{c}
v_{1}\\ v_{2}\\ v_{3}
\end{array}
\right ):= \left ( 
\begin{array}{ccc}
0 & -v_{3} & v_{2} \\
v_{3} & 0 & -v_{1}\\
-v_{2} & v_{1} & 0
\end{array}
\right ).
$$
Consequently, we have   $\mskw(v)w = v\times w$ for $v,w\in\V$.  We also define
$\vskw=\mskw^{-1}\circ \skw: \M\to \V$. We further define the Hessian operator $\hess:=\grad\circ\grad$ and for any matrix-valued function $u$, $\mathcal{S}u:=u^{T}-\tr(u)I$. In our convention, the proxies of exterior derivatives ($\grad$, $\curl$, $\div$, etc.) act column-wise, and so do the Poincar\'e operators for the de~Rham complexes.

 Some examples in three dimensions can be demonstrated by the following diagram in vector/matrix forms:
 \begin{equation}\label{diagram-4rows}
\begin{tikzcd}
0 \arrow{r} &H^{q}\otimes \mathbb{R}  \arrow{r}{\grad} &H^{q-1}\otimes \mathbb{V} \arrow{r}{\curl} &H^{q-2}\otimes \mathbb{V} \arrow{r}{\div} & H^{q-3}\otimes \mathbb{R} \arrow{r}{} & 0\\
0 \arrow{r}&H^{q-1}\otimes \mathbb{V}\arrow{r}{\grad} \arrow[ur, "I"]&H^{q-2}\otimes \mathbb{M}  \arrow{r}{\curl} \arrow[ur, "2\vskw"]&H^{q-3}\otimes \mathbb{M} \arrow{r}{\div}\arrow[ur, "\tr"] & H^{q-4}\otimes \mathbb{V} \arrow{r}{} & 0\\
0 \arrow{r} &H^{q-2}\otimes \mathbb{V}\arrow{r}{\grad} \arrow[ur, "-\mskw"]&H^{q-3}\otimes \mathbb{M}  \arrow{r}{\curl} \arrow[ur, "\mathcal{S}"]&H^{q-4}\otimes \mathbb{M} \arrow{r}{\div}\arrow[ur, "2\vskw"] & H^{q-5}\otimes \mathbb{V} \arrow{r}{} & 0\\
0 \arrow{r} &H^{q-3}\otimes \mathbb{R}\arrow{r}{\grad} \arrow[ur, "\iota"]&H^{q-4}\otimes \mathbb{V}  \arrow{r}{\curl} \arrow[ur, "-\mskw"]&H^{q-5}\otimes \mathbb{V} \arrow{r}{\div}\arrow[ur, "I"] & H^{q-6}\otimes \mathbb{R} \arrow{r}{} & 0.
 \end{tikzcd}
\end{equation}
This diagram should not be confused with \eqref{diag:multi-rows}, as in \eqref{diagram-4rows}, we only read two rows at one time. From the first two rows, we get the Hessian complex; the middle two rows give the elasticity complex; and the last two rows give the $\div\div$ complex. See \cite{arnold2021complexes} for more details.  

\textbf{The elasticity complex}: In three space dimensions, the diagram for deriving the elasticity complex and the Poincar\'e operators is the two middle rows in \eqref{diagram-4rows}.
We first discuss the case with smooth functions. The benefit is that we get $P\circ P=0$ if we use the standard Poincar\'e operators for the de~Rham complex (which may not hold for the regularized version \cite{costabel2010bogovskiui}).

In the diagram, $S^{0}=-\mskw$ and $S^{1}=\mathcal{S}$ are injective. Therefore the second components of $Y^{0}$ and $Y^{1}$ are zero (in the kernel of $S^{0}$ and $S^{1}$, respectively). Then the operator $\mathscr{P}^{1}$ at index one essentially maps between the first components with the form $P-P(PS-SP)Td$, which further simplifies to $P+PSPTd$ if 
$P\circ P=0$.

  Let $e\in C^{\infty}\otimes \mathbb{S}$. Then $\mathscr{P}^{1}(e)\in C^{\infty}\otimes \mathbb{V}$ is given by 
\begin{align*} 
\mathscr{P}^{1}(e)&=\int_{0}^{x}e(y)\cdot\,dy-\int_{0}^{x}(\underbrace{\mskw\int_{0}^{y}\underbrace{(e({z})\times \nabla)^{T}}_\text{$Td(e)$}\cdot dz)}_\text{$SPTd(e)$}\cdot dy.
\end{align*}
Here we have used the de~Rham Poincar\'e operators with smooth forms \eqref{smooth-poincare} for the definition of $P$. We can verify one of the null homotopy identities by straightforward calculations. Indeed, 
if $e=\deff(w)$ for some vector-valued function $w$, then  
\begin{align*}
\mathscr{P}^{1}(e) &=\int_{0}^{x}e(y)\,dy-\frac{1}{2}\int_{0}^{x}\mskw (\int_{0}^{y}(\nabla\times w\nabla)\cdot\,dz)\cdot\, dy\\
&=\int_{0}^{x}e(y)\,dy-\frac{1}{2}\int_{0}^{x}\mskw (\nabla\times w(y) -\nabla\times w(0)) \cdot\, dy
\end{align*}
Note that 
\begin{align*}
\int_{0}^{x}\deff(w)\,dy&=\int_{0}^{x}(w\nabla)\,dy+\frac{1}{2}\int_{0}^{x}(\nabla w-w\nabla)\cdot\,dy\\
&=w(x)-w(0)+\frac{1}{2}\int_{0}^{x}\mskw \nabla\times w (y)\cdot\, dy.
\end{align*}
This gives 
$$
\mathscr{P}^{1}(\deff(w))=w(x)-w(0)+\frac{1}{2}x\wedge (\nabla\times w(0)).
$$
Since the right hand side is an infinitesimal rigid motion, it lies in the kernel of $\deff$, so the right null homotopy relation is satisfied in degree zero. This reflects the fact that the original sum of de Rham complexes has a six dimensional space of constants as its cohomology in degree zero.

In the case of general Sobolev spaces, the above smooth version generalized to $\mathscr{P}^{1}=P(PS-SP)Td$, where the de~Rham Poincar\'e operators $P^{\bs}$ are replaced by the regularized version in \cite{costabel2010bogovskiui}.

At index two, the formula simplifies to $\mathscr{P}^{2}=-{\Pi_{\ran^{\perp}}}P(PS-SP)$, which maps the second component of $Y^{2}$ to the first component of $Y^{1}$  as $S^{1}$ is bijective. At index three, we get $\mathscr{P}^{3}=-{\Pi_{\ker}}dTP(PS-SP)+{\Pi_{\ker}}P$. If $P\circ P=0$, these formulas further simplify to $\mathscr{P}^{2}={\Pi_{\ran^{\perp}}}PSP$ and $\mathscr{P}^{3}={\Pi_{\ker}}dTPSP+{\Pi_{\ker}}P$, respectively.

 Compared to the Ces\`aro-Volterra formula with little regularity  \cite{ciarlet2010cesaro}, the new formula $\mathscr{P}^{1}$ is polynomial-preserving (when the input de~Rham $P^{\bs}$ operators are so) and has a more explicit form (whereas the formula in \cite{ciarlet2010cesaro} is given by duality), and is valid for a wide range of spaces, following the results in \cite{costabel2010bogovskiui}. To see that the two results are different, we note that the formula in \cite{ciarlet2010cesaro} boils down to the classical Ces\`aro-Volterra formula if the fields it acts on are smooth enough. Contrarily, this is not the case for the $\mathscr{P}^{1}$ operator derived above as the input $P^{\bs}$ involves smoothing (averaging) even for smooth functions. 

\begin{remark}
In the proxies above, for the convenience of notation, we have used the convention that differential operators act row-wise on a matrix-valued function. For example, if $u$ is a vector-valued 0-form, then $du$ corresponds to $(du)_{ij}:=\partial_{j}u_{i}$, where $(du)_{ij}$ denotes the $(i, j)$-th entry of the matrix proxy of $du$. Alternatively, in another notation, $du$ corresponds to $u\nabla$, where $\nabla$ acts ``from the right'' (row-wise).  Correspondingly, the de~Rham Poincar\'e operators act row-wise. For example, if $v$ is a matrix-valued function, then $\int_{0}^{x}dy\wedge v(y)$ extends  \eqref{dy-wedge} to  each row of $v$. This is consistent with the convention in \cite{christiansen2020poincare}, but different from \cite{arnold2021complexes,vcap2022bgg}.
\end{remark}

 \medskip
 
 \textbf{The twisted complex}: The vector/matrix form of the elasticity version of the twisted
 complex (obtained from the two middle rows in \eqref{diagram-4rows})
 is the following:
  \begin{equation} \label{twisted-elasticity}
\begin{tikzcd}[ampersand replacement=\&, column sep=1in]
\left (
\begin{array}{c}
{C^{\infty}\otimes \mathbb{V}}\\
{C^{\infty}\otimes \mathbb{V}}
\end{array}
\right ) \arrow{r}{\left (
\begin{array}{cc}
\grad & \mskw\\
0&\grad
\end{array}
\right ) } \&\left (
\begin{array}{c}
{C^{\infty}\otimes \mathbb{M}}\\
{C^{\infty}\otimes \mathbb{M}} 
\end{array}
\right )
\arrow{r}{\left (
\begin{array}{cc}
\curl & -\mathcal{S}\\
0&\curl
\end{array}
\right ) } \&\left (
\begin{array}{c}
{C^{\infty}\otimes \mathbb{M}}  \\
{C^{\infty}\otimes \mathbb{M}} 
\end{array}
\right )\arrow{r}{\left (
\begin{array}{cc}
\div & -2\vskw\\
0&\div
\end{array}
\right ) } \&\left (
\begin{array}{c}
{C^{\infty}\otimes \mathbb{V}}  \\
{C^{\infty}\otimes \mathbb{V}} 
\end{array}
\right ).
 \end{tikzcd}
\end{equation}
The Poincar\'e operators have the general form (see \eqref{PV})
$$
P_{V}=\left (
\begin{array}{cc}
P &-P(PS-SP)\\
0 & P
\end{array}
\right), 
$$
which further simplifies to 
$$
\left (
\begin{array}{cc}
P & -PSP\\
0 & P
\end{array}
\right),
$$
if the complex property $P\circ P=0$ holds for the input. 

For the   complex with smooth fields (for which the classical de~Rham Poincar\'e operators \eqref{smooth-poincare} can be used),  vector-matrix proxies have the form (see \eqref{dy-dot}-\eqref{dy-tensor} for the notation):
$$
P_{V}^{1}\left (
\begin{array}{c}
u\\
w
\end{array}
\right )=
\left (
\begin{array}{c}
\int_{0}^{x}u(y)\cdot\,dy+\int_{0}^{x}(-\mskw\int_{0}^{y}w(z)\cdot\,dz)\cdot\,dy\\
\int_{0}^{x}w(y)\cdot\,dy
\end{array}
\right )=\left (
\begin{array}{c}
\int_{0}^{x}u(y)\cdot\,dy+\int_{0}^{x}( \int_{0}^{y}w(z)\cdot\,dz)\wedge dy\\
\int_{0}^{x}w(y)\cdot\,dy
\end{array}
\right ),
$$
$$
P_{V}^{2}\left (
\begin{array}{c}
u\\
w
\end{array}
\right )=
\left (
\begin{array}{c}
\int_{0}^{x}u(y)\wedge dy+\int_{0}^{x}(\mathcal{S}\int_{0}^{y}w(z)\wedge dz)\wedge dy\\
\int_{0}^{x}w(y)\wedge dy
\end{array}
\right ),
$$
and
$$
P_{V}^{3}\left (
\begin{array}{c}
u\\
w
\end{array}
\right )=
\left (
\begin{array}{c}
\int_{0}^{x}u(y)\otimes dy+\int_{0}^{x}(2\vskw\int_{0}^{y}w(z)\otimes dz)\otimes dy\\
\int_{0}^{x}w(y)\otimes dy
\end{array}
\right ),
$$
respectively.

 \subsection{Polynomial complexes}\label{sec:polynomial}

The examples in \cite{vcap2022bgg} have vector-valued de~Rham complexes as input (each row in \eqref{diag:multi-rows}). Then all the operators $d^{\bs}$, $P^{\bs}$ (regularized Poincar\'e operators for the de~Rham complex) and $S^{\bs}$ are polynomial-preserving. Following the discussions in Section \ref{sec:Polynomial-preservation}, we obtain polynomial exact sequences. Examples include the elasticity complex,
   \begin{equation}\label{poly-elast} 
\begin{tikzcd}
  0 \arrow{r}& \mathcal{P}_{r}\otimes \mathbb{V}\arrow{r}{\deff}& \mathcal{P}_{r-1}\otimes \mathbb{S}\arrow{r}{\inc }& \mathcal{P}_{r-3}\otimes \mathbb{S} \arrow{r}{\div }&\mathcal{P}_{r-4}\otimes \mathbb{V} \arrow{r}&0,
\end{tikzcd}
\end{equation}
the conformal Hessian complex
   \begin{equation} \label{poly-conf-hess}
\begin{tikzcd}
  0 \arrow{r}& \mathcal{P}_{r} \arrow{r}{\dev\hess}& \mathcal{P}_{r-2}\otimes (\mathbb{S}\cap \mathbb{T})\arrow{r}{\sym\curl }& \mathcal{P}_{r-3}\otimes  (\mathbb{S}\cap \mathbb{T}) \arrow{r}{\div\div }&\mathcal{P}_{r-5}\arrow{r}&0,
\end{tikzcd}
\end{equation}
and the conformal deformation complex
   \begin{equation} \label{poly-conf-def}
\begin{tikzcd}
 0 \arrow{r}& \mathcal{P}_{r}\otimes \mathbb{V}  \arrow{r}{\dev\deff}& \mathcal{P}_{r-1}\otimes (\mathbb{S}\cap \mathbb{T})\arrow{r}{\cot }& \mathcal{P}_{r-4}\otimes  (\mathbb{S}\cap \mathbb{T}) \arrow{r}{ \div }&\mathcal{P}_{r-5}\otimes \mathbb{V} \arrow{r}&0.
\end{tikzcd}
   \end{equation}
\begin{theorem}
The polynomial complexes  \eqref{poly-elast}, \eqref{poly-conf-hess} and \eqref{poly-conf-def} are exact in degrees $>0$ while the cohomologies in dimension $0$ are the $6$-dimensional space of rigid infinitesimal motions, the $5$-dimensional space $\ker(\dev\hess)$ and the $10$ dimensional space $\ker(\dev\deff)$, respectively.
\end{theorem}
 \begin{proof}
On $\Bbb R^n$, we get the Poincar\'e operators on the smooth complexes as discussed above and they are all polynomial preserving. Thus we get the claimed information on the cohomology.  
\end{proof}
For the three examples derived from \eqref{diagram-4rows} (the elasticity, Hessian and $\div\div$ complexes), polynomial complexes and their exactness were discussed in \cite{chen2022discrete,chen2022finite-divdiv,chen2022finite,christiansen2020poincare}. 

The above complexes are analogues of \eqref{polynomial-1} with full polynomials. Another family of polynomial complexes, aimed at an analogue of \eqref{polynomial-2}, can be derived from a sequence of polynomial spaces of equal degree, e.g., 
   \begin{equation}\label{polynomial-elas-r} 
\begin{tikzcd}
  0 \arrow{r}& \mathcal{P}_{r}\otimes \mathbb{V}\arrow{r}{\deff}& \mathcal{P}_{r}\otimes \mathbb{S}\arrow{r}{\inc }& \mathcal{P}_{r}\otimes \mathbb{S} \arrow{r}{\div }&\mathcal{P}_{r}\otimes \mathbb{V} \arrow{r}&0.
\end{tikzcd}
\end{equation}
Here we used the elasticity complex as an example, but the discussions also hold for other complexes.  Clearly, \eqref{polynomial-elas-r} is a complex, but we do not get Poincar\'e operators here, since they do not preserve the degree of polynomials and hence we cannot make conclusions on the cohomology. To fix this, we can modify the spaces in the complex by including a $\mathscr{P}^\bs$-component, i.e., 
\begin{equation}\label{polynomial-elas-r+} 
\begin{tikzcd}[row sep=small]
   0 \arrow{r}& \mathcal{P}_{r}\otimes \mathbb{V} + \mathscr{P}^{1}[\mathcal{P}_{r}\otimes \mathbb{S}]\arrow{r}{\deff}& \mathcal{P}_{r}\otimes \mathbb{S} + \mathscr{P}^{2}[\mathcal{P}_{r}\otimes \mathbb{S}]\\
   \arrow{r}{\inc }& \mathcal{P}_{r}\otimes \mathbb{S} + \mathscr{P}^{3}[\mathcal{P}_{r}\otimes \mathbb{V}] \arrow{r}{\div }&\mathcal{P}_{r}\otimes \mathbb{V} \arrow{r}&0.
\end{tikzcd}
\end{equation}
 The null homotopy relations for the operators $\mathscr{P}^{i}$ easily imply that the operators in \eqref{polynomial-elas-r+} indeed map each space in the sequence to the next space. In general, it is still not clear that we get a Poincar\'e operator for this complex, since the spaces are not necessarily mapped to each other by the operators $\mathscr{P}^i$. This is true here, however, since we can use the Poincar\'e operators derived from the smooth de~Rham version \eqref{smooth-poincare} which satisfy
 the complex property. Then \eqref{polynomial-elas-r+} is also a complex with $\mathscr{P}^{\bs}$ and we get the conclusions on the cohomology.

This construction above is a general procedure. We can start with any complex of $d^{\bs}$ and complete it by adding $\mathscr{P}^{\bs}$ components such that the sequence is also a complex of $\mathscr{P}^{\bs}$. 
One may further investigate the construction of finite elements based on polynomial exact sequences, but this is beyond the scope of this paper. 


\section*{Acknowledgement}

A.\v C.\ is supported by the Austrian Science Fund (FWF): P33559-N. {This article/publication is based upon work from COST Action 
CaLISTA CA21109 supported by COST (European Cooperation in Science and Technology). www.cost.eu.} K.H. is supported by a Royal Society University Research Fellowship (URF$\backslash${\rm R}1$\backslash$221398).

\bibliographystyle{siam}      
\bibliography{reference}{}   

\end{document}